\documentclass[11pt,a4paper]{article}

\usepackage[utf8]{inputenc}
\usepackage[T1]{fontenc}
\usepackage{lmodern}
\usepackage{microtype}

\usepackage{amsmath,amssymb,amsthm,mathtools}
\usepackage{enumitem}
\usepackage[margin=1in]{geometry}
\usepackage[title]{appendix}

\usepackage{xcolor}
\usepackage[colorlinks=true,linkcolor=blue!55!black,citecolor=blue!55!black,%
            urlcolor=blue!55!black]{hyperref}
\usepackage{cleveref}
\usepackage{aliascnt}

\theoremstyle{plain}
\newtheorem{theorem}{Theorem}[section]
\newaliascnt{proposition}{theorem}
\newtheorem{proposition}[proposition]{Proposition}
\aliascntresetthe{proposition}
\newaliascnt{lemma}{theorem}
\newtheorem{lemma}[lemma]{Lemma}
\aliascntresetthe{lemma}
\newaliascnt{corollary}{theorem}
\newtheorem{corollary}[corollary]{Corollary}
\aliascntresetthe{corollary}
\theoremstyle{definition}
\newaliascnt{definition}{theorem}
\newtheorem{definition}[definition]{Definition}
\aliascntresetthe{definition}
\newaliascnt{example}{theorem}
\newtheorem{example}[example]{Example}
\aliascntresetthe{example}
\newaliascnt{construction}{theorem}

\aliascntresetthe{construction}
\theoremstyle{remark}
\newaliascnt{remark}{theorem}
\newtheorem{remark}[remark]{Remark}
\aliascntresetthe{remark}
\newaliascnt{notation}{theorem}
\newtheorem{notation}[notation]{Notation}
\aliascntresetthe{notation}
\newaliascnt{conjecture}{theorem}

\aliascntresetthe{conjecture}

\crefname{theorem}{theorem}{theorems}
\Crefname{theorem}{Theorem}{Theorems}
\crefname{proposition}{proposition}{propositions}
\Crefname{proposition}{Proposition}{Propositions}
\crefname{lemma}{lemma}{lemmas}
\Crefname{lemma}{Lemma}{Lemmas}
\crefname{corollary}{corollary}{corollaries}
\Crefname{corollary}{Corollary}{Corollaries}
\crefname{definition}{definition}{definitions}
\Crefname{definition}{Definition}{Definitions}
\crefname{example}{example}{examples}
\Crefname{example}{Example}{Examples}
\crefname{construction}{construction}{constructions}
\Crefname{construction}{Construction}{Constructions}
\crefname{remark}{remark}{remarks}
\Crefname{remark}{Remark}{Remarks}
\crefname{notation}{notation}{notations}
\Crefname{notation}{Notation}{Notations}
\crefname{conjecture}{conjecture}{conjectures}
\Crefname{conjecture}{Conjecture}{Conjectures}

\newcommand{\RR}{\mathbb{R}}
\newcommand{\CC}{\mathbb{C}}
\newcommand{\Id}{\mathrm{I}}
\newcommand{\diag}{\operatorname{diag}}
\newcommand{\Tr}{\operatorname{Tr}}
\newcommand{\spec}{\operatorname{spec}}
\DeclareMathOperator{\OO}{O}
\DeclareMathOperator{\UU}{U}
\DeclareMathOperator{\SO}{SO}

\newcommand{\hP}{\widehat P}
\newcommand{\hQ}{\widehat Q}
\newcommand{\hD}{\widehat D}
\newcommand{\hL}{\widehat L}
\newcommand{\hT}{\widehat T}
\newcommand{\hpi}{\widehat\pi}
\newcommand{\hOm}{\widehat\Omega}
\newcommand{\Bc}{B}
\newcommand{\HNUO}{H_{\mathrm{NUO}}}
\newcommand{\Asym}{A_{\mathrm{sym}}}
\newcommand{\Ksym}{K_{\mathrm{sym}}}
\newcommand{\gQ}{\widetilde Q}

\newcommand{\sigmax}{\sigma^x}
\newcommand{\sigmay}{\sigma^y}
\newcommand{\sigmaz}{\sigma^z}
\newcommand{\Jx}{J_x}
\newcommand{\Jy}{J_y}
\newcommand{\Jz}{J_z}
\newcommand{\nhat}{\hat n}
\newcommand{\zhat}{\hat z}

\newcommand{\ket}[1]{\lvert #1\rangle}
\newcommand{\bra}[1]{\langle #1\rvert}
\newcommand{\braket}[2]{\langle #1\vert #2\rangle}
\newcommand{\ketbra}[2]{\lvert #1\rangle\langle #2\rvert}

\newcommand{\Quad}{\mathcal{Q}}
\newcommand{\BCo}{\mathrm{BC}}
\newcommand{\Fishinfo}{\mathcal{F}}
\newcommand{\angleFR}{\angle_{\mathrm{FR}}}
\newcommand{\Hel}{\mathrm{H}}
\newcommand{\Sym}{\operatorname{Sym}}
\newcommand{\kap}{\kappa}

\newcommand{\sFR}{\sigma^{\mathrm{FR}}}

\newcommand{\Ephase}{\mathcal{E}}   % diagonal phase gauge, was \widehat\Lambda
\newcommand{\Wt}{\mathsf{w}}        % mixture weight matrix, was W

\title{The pseudo-quantum representation of \\finite reversible Markov chains}

\author{%
  Eduardo J. Neves\\[2pt]
  \normalsize Instituto de Matem\'atica e Estat\'istica\\
  \normalsize Universidade de S\~ao Paulo\\
  \normalsize Rua do Mat\~ao 1010, 05508-090 S\~ao Paulo, SP, Brazil\\
  \normalsize \texttt{neves@ime.usp.br}%
}

\date{July 31, 2026}

\begin{document}

\maketitle

\begin{abstract}
\noindent
The pseudo-quantum representation re-encodes a finite, irreducible, reversible
continuous-time Markov chain with $M$ states as a complex-orthogonal flow on the
doubled space $\CC^{2}\otimes\CC^{M}$. After uniformisation, a square-root gauge,
a doubling of the state space and a diagonal unitary twist, the chain generates
an entire one-parameter group $W_z=e^{zK}$, $z\in\CC$. This single group is the
object of the paper. Its real slice $z=s\in\RR$ reproduces the stochastic
semigroup exactly through an explicit decoding whose probabilistic meaning is
the parity of the number of ticks of the uniformised chain. Its imaginary slice
$z=i\theta$ is a finite-dimensional unitary quantum system. The Markov chain and
the quantum system are therefore not two analogous models but two restrictions
of one entire representation, the chain read along the real axis and the quantum
objects read along the imaginary axis. The setup also produces, with no further
input, a pseudo-Schr\"odinger equation, a bilinear von Neumann equation for a
complex-symmetric pseudo-density, and a pseudo-Bloch vector equation on a
non-compact quadric. We develop the construction from first principles and work
out one model completely, the usual symmetric Ehrenfest urn. Its gauged
generator equals $\tfrac2n J_x$, the spin-$n/2$ operator, its relaxation modes
are Lorentz boosts, and its imaginary slice is a depth-one quantum circuit whose
Born distribution is the classical urn law under the clock
$\theta(t)=\arccos(e^{-2t})$. The same clock identifies the Fisher--Rao lift of
the urn trajectory with a rigid spin-coherent-state orbit. This is a preliminary simplified version of a longer paper. It treats only finite
state spaces and only the symmetric Ehrenfest urn, and it previews without proofs the
asymmetric urn, two queues, the symmetric simple exclusion process and the
stochastic Ising model. Short appendix primers make the finite-dimensional
quantum, geometric and information-theoretic language self-contained.
\end{abstract}

% =====================================================================
%  KEYWORDS AND CLASSIFICATION CODES
% =====================================================================

\begin{center}
\begin{minipage}{0.92\textwidth}

\smallskip

\end{minipage}
\end{center}

\bigskip

{\small\tableofcontents}

\bigskip

\section{Introduction}
\label{sec:intro}

\subsection{A reversible shadow of a dissipative chain}
\label{sub:intro-shadow}

A finite, irreducible, reversible continuous-time Markov chain is one of the
simplest models of stochastic dynamics. A particle hops among states
$\Omega=\{1,\dots,M\}$ at exponential rates, the column vector law $\pi_t$ obeys the
forward equation $\dot\pi_t=Q^{T}\pi_t$, and from any start the law relaxes to a
unique equilibrium $\nu$. The evolution is dissipative. The semigroup
$T_t=e^{tQ^{T}}$ contracts, the relative entropy with respect to $\nu$ decreases,
and the spectral gap fixes the relaxation time. One cannot run the semigroup
backwards to recover the initial law from a long observation in any stable way,
because information about the start is steadily lost. Standard references for this
picture are Norris \cite{Norris}, Levin, Peres and Wilmer \cite{LPW}, and Liggett
\cite{Liggett_ctmc}.

A closed quantum system is the opposite kind of object. It evolves by a unitary
group $e^{-i\theta H}$ that conserves total probability, is exactly reversible, and
forgets nothing. Its generator $H$ is Hermitian, its orbits are bounded, and the
motion is a rigid rotation of a state vector in a Hilbert space \cite{Sakurai}.

These two pictures look incompatible. The purpose of this paper is to show that for
a reversible chain they are two faces of a single object. We build a single
finite-dimensional, deterministic, complex-linear flow
\begin{equation}
W_z=e^{zK},\qquad z\in\CC,
\label{eq:intro-flow}
\end{equation}
which is an entire one-parameter group of complex matrices, holomorphic on the
whole plane. This entirety property is already a departure from ordinary quantum
mechanics. The standard quantum formalism is sesquilinear, built on the Hermitian
inner product $\Phi^{*}\Psi$, and a Hermitian-conjugate object depends on $\bar z$
and is therefore not holomorphic in a complex time. The one change that makes
\eqref{eq:intro-flow} entire is the replacement of that sesquilinear pairing by the
\emph{bilinear} pairing $\Phi^{T}\Psi$, which carries no complex conjugation. This
bilinear versus sesquilinear distinction is the crux of the whole construction
(\cref{sec:pseudostate} and the primer in \cref{sec:primer-quantum}).

The group $W_z$ has two distinguished lines. On the imaginary axis, $z=i\theta$,
the matrices $W_{i\theta}$ are ordinary unitary rotations, a genuine unitary
quantum evolution $e^{-i\theta H}$ on a finite Hilbert space canonically attached
to the chain. On the real axis, $z=s\in\RR$, the matrices $W_s$ are
\emph{hyperbolic} rotations. They are complex-orthogonal but not unitary, and we
call them the \emph{non-unitary orthogonal} (NUO) flow. On each relaxation mode of
the chain $W_s$ acts as a two-dimensional Lorentz boost, that is a hyperbolic
rotation of signature $(1,1)$, and it is this real slice that reproduces the
dissipative Markov semigroup $T_t$ exactly, after an explicit decoding described
below. We call the whole construction the \emph{pseudo-quantum representation}
(PQR) of the chain.

The point to hold on to is that there is only one object here. The chain is not
being compared with a quantum system, and the quantum system is not being
extracted from the chain by a further modelling step. One generator $K$ is built
from the chain by the four moves listed above, one entire group $W_z=e^{zK}$ is
formed, and then the classical and the quantum content are two restrictions of
that single representation. Restricting to the real axis and decoding gives back
$T_t$ with nothing lost. Restricting to the imaginary axis gives the unitary
group $e^{-i\theta H}$, its Hermitian Hamiltonian, its energy levels, its Born
rule and its bounded recurrent orbits. Holomorphy is what ties the two together,
since an entire function is determined by its restriction to any line, so the
Markov data and the quantum data determine each other. Dissipation and unitary
reversibility are not in conflict in this picture. They are what one sees looking
along two different directions in the domain of the same matrix-valued entire
function.

The setup also hands us, at no extra cost, a state formalism that mirrors quantum
mechanics with the same single change of pairing. A complex-symmetric
\emph{pseudo-density} $\rho_z=\Phi_z\Phi_z^{T}$ obeys a \emph{bilinear von
Neumann equation}
\begin{equation}
\dot\rho_z=[K,\rho_z],
\label{eq:intro-vn}
\end{equation}
identical in form to the quantum von Neumann equation $\dot\sigma=-i[H,\sigma]$
with $K$ in the role of $-iH$, the pseudo-wave $\Phi_z$ obeys a
\emph{pseudo-Schr\"odinger equation} $\dot\Phi_z=K\Phi_z$, and the coordinates of
$\rho_z$ in a basis of traceless complex-symmetric matrices obey a
\emph{pseudo-Bloch vector equation} $\dot r=\Omega\,r$ with
$\Omega^{T}=-\Omega$. On the imaginary slice these three reduce to the genuine
Schr\"odinger, von Neumann and Bloch equations of quantum mechanics. On the real
slice they describe a non-compact hyperbolic motion on a complex affine quadric.
We derive all three in \cref{sec:pseudostate}. We regard them as a bonus of the
construction rather than as one of its results. They come out of the bilinear
pairing by themselves, they are exact, and what they are good for is open. The
non-compact quadric on which the pseudo-Bloch vector moves, the meaning of its
conserved bilinear length when that length is complex, and the classification of
its orbits are all questions this paper raises and does not settle.

\subsection{The decoding, and why nothing is lost}
\label{sub:intro-decode}

The encoding is lossless in a concrete sense. Everything in this paper
is finite-dimensional linear algebra over $\CC$, and the recovery of $\pi_t$ from
$W_s$ is an exact identity, not an approximation or a scaling limit. The decoding
has three steps, a fixed similarity transformation, a scalar damping factor
$e^{-s}$, and a sum over a doubled index. The probabilistic content of these steps
is elementary. The doubling sorts trajectories by the parity of the number of
jumps they have made, and the damping is the normalisation into probabilities. We
make this precise in \cref{sec:dilation,sec:decode}, where the recovery identity is
stated and proved.

\subsection{Why the representation is worth having}
\label{sub:intro-why}

The construction is useful for three reasons that recur throughout the paper.

First, it turns relaxation into geometry. On a relaxation mode of rate $\gamma$
the real-time flow acts as a Lorentz boost whose rapidity grows linearly in the
rescaled time. The approach to equilibrium is the boost running off toward
infinite rapidity, while the scalar damping renormalises the growing amplitude
back to a probability. The hyperbolic angle through which a mode has turned is a
clean reparametrisation of how far the chain has relaxed. This is developed in
\cref{sec:ehrenfest}.

Second, it yields an exact dictionary, rather than an analogy, between a
reversible chain and a unitary system. The square-root gauge that maps a
reversible generator to a symmetric operator is classical and widely used in the
analysis of reversible chains \cite{Aldous,LPW}. What the PQR adds is that the
gauged chain and a true unitary model are the two slices of one entire group of
complex-orthogonal matrices, so a statement about either slice is a statement
about the same generator $K$ and constrains the other slice. In
\cref{sec:concluding} we separate the finite-state construction from the
additional identities that arise for the symmetric Ehrenfest urn.

Third, it is elementary and self-contained. It requires only linear algebra, the
matrix exponential,
and the basic vocabulary of Markov chains. The small
amount of quantum, geometric, and information-theoretic language we use is
collected in three short appendix primers
(\cref{sec:primer-quantum,sec:primer-geometry,sec:primer-info}).

\subsection{The Ehrenfest urn, worked out in detail}
\label{sub:intro-ehrenfest}

The Ehrenfest urn chain, in which $n$ balls move one at a time between two urns,
is the textbook model of relaxation and recurrence, introduced to reconcile
Boltzmann's H-theorem with Poincar\'e recurrence \cite{Kac}. It
is the single example this paper works out completely, because of one pleasant
property. The gauged generator of the symmetric Ehrenfest chain is exactly
$\tfrac2n$ times the spin-$n/2$ angular-momentum matrix $J_x$. From this single
identity the whole apparatus becomes elementary spin algebra, every object of the
general theory can be computed in closed form, and three exact connections appear.

The first connection is to special relativity. Each relaxation mode is a Lorentz
boost, and the spectrum of relaxation rates is a spectrum of rapidity rates. The
second is computational. The imaginary slice is an honest unitary on a finite
Hilbert space, and for the Ehrenfest chain it reduces to a \emph{depth-one quantum
circuit} whose Born-rule sampling distribution is the classical law $\pi_t$.
 The third is the Majorana stellar picture \cite{Majorana,Radcliffe}. A relaxing distribution
becomes a rigid constellation of $n$ points on the Bloch sphere, coincident at the
north pole at time zero, gliding together down a meridian toward the equator.
These three descriptions are tied to the same spin representation and the same
nonlinear clock.

A single transcendental quantum clock ties the classical and quantum readings
together. The chain at time $t$ matches the spin rotated through the angle
\begin{equation}
\theta(t)=\arccos\bigl(e^{-2t}\bigr),
\label{eq:intro-clock}
\end{equation}
which we can translate in two ways, spectrally through the subdominant Markov eigenvalue
and information-geometrically as a Fisher--Rao statistical angle.

\subsection{Scope of this preliminary version}
\label{sub:intro-scope}

This is a preliminary and deliberately simplified account of the pseudo-quantum
representation, and we state at the outset what has been left out. A longer
paper, currently in preparation, will contain the general theory and four
further families of models. The present version keeps two restrictions
throughout.

The first restriction is that the state space is finite. Every object in this
paper is a matrix of finite size and every identity is an identity between such
matrices, so the entire argument is finite-dimensional linear algebra. No
functional analysis is needed, no domain question arises, and no limit is taken.
The infinite-state case is genuinely harder, because a chain with unbounded
rates has no uniformisation and its gauged generator is an unbounded self-adjoint
operator, and it is deferred. A simple example where the usual uniformisation
strategy does not work is the queue $M/M/\infty$.

The second restriction is that the single worked model is the usual
\emph{symmetric} Ehrenfest urn, in which the two jump rates are equal. This is
the smallest model in which every feature of the construction is visible and
every object is available in closed form, and it keeps the spin algebra as
simple as it can be, namely rotations about one fixed axis. The asymmetric Ehrenfest urn,
in which the quantum spin axis tilts, is deferred with the rest.

The organisation follows that plan. After the notation and the square-root gauge
in \cref{sec:setup}, \cref{sec:dilation} builds the complex-orthogonal dilation
and states its defining property. \Cref{sec:decode,sec:pseudostate,sec:imaginary}
prove the real-slice decoding, develop the bilinear state formalism, and identify
the unitary imaginary slice, all for a general finite reversible chain.
\Cref{sec:ehrenfest} specialises everything to the symmetric Ehrenfest urn. It
contains the spin identity, the explicit spectral (Krawtchouk) decoding, the Born
rule under the quantum clock, the smallest instances as explicit matrices, and
the three connections above. \Cref{sec:geometry-info} records the geometric and
information-theoretic structure of the representation, including the way it
separates the Fisher--Rao geometry of statistical inference from a flat
$\chi^{2}$ geometry of convergence to equilibrium, and the conserved spectral
invariants of mixtures.

\Cref{sec:teaser} is a preview. It describes, in words and without proofs, what
the construction produces for the four families excluded here, namely the
asymmetric Ehrenfest urn and its tilted spin, the two textbook reversible
queues, the symmetric simple exclusion process, and the stochastic Ising model.
It states no lemma and proves nothing. Its purpose is to indicate the range of
the construction and to say what the complete paper will contain.
\Cref{sec:concluding} collects the contributions made
within the present finite-state scope. Appendix primers on the quantum and Lie
language (\cref{sec:primer-quantum}), on the differential-geometric language
(\cref{sec:primer-geometry}) and on the information-theoretic language
(\cref{sec:primer-info}) make the paper accessible to a reader from probability
with no background in these areas.

% ======================================================================

\section{Reversible Markov chains and the square-root gauge}
\label{sec:setup}

This section sets notation and recalls the classical tool on which the whole
construction rests, the square-root gauge that turns a reversible generator into a
symmetric matrix. The generator and semigroup formalism and the forward equation
are in Norris \cite[Ch.~2--3]{Norris}. Uniformisation is treated in
\cite[\S2.1]{Norris}. Detailed balance and the
symmetrising similarity are classical, see Aldous and Fill \cite[Ch.~3]{Aldous} and
Levin, Peres and Wilmer \cite[\S12.1]{LPW}. The present simplified paper uses this gauge
entirely within the finite-dimensional setting.

\subsection{Generator, semigroup, stationarity}
\label{sub:setup-gen}

Let $\Omega=\{1,\dots,M\}$ with $M\ge2$, and let $Q=(Q_{ij})$ be the generator of an
irreducible continuous-time Markov chain in the row convention,
\begin{equation}
Q_{ij}\ge0\ (i\ne j),\qquad Q_{ii}=-\sum_{j\ne i}Q_{ij},\qquad Q\mathbf 1=0,
\label{eq:gen}
\end{equation}
where $\mathbf 1=(1,\dots,1)^{T}$. The entry $Q_{ij}$ is the rate of jumps $i\to j$,
and $-Q_{ii}$ is the total exit rate from $i$. We write laws as column vectors
$\pi_t\in\RR^{M}$, so the forward equation and its solution read
\begin{equation}
\dot\pi_t=Q^{T}\pi_t,\qquad \pi_t=T_t\pi_0,\qquad T_t=e^{tQ^{T}}.
\label{eq:forward}
\end{equation}
Irreducibility gives a unique strictly positive stationary law $\nu$, normalised by
$\sum_i\nu_i=1$ and characterised by $Q^{T}\nu=0$. From every start $\pi_t\to\nu$ as
$t\to\infty$.

\begin{definition}[Reversibility and detailed balance]
\label{def:reversible}
The chain is \emph{reversible} with respect to $\nu$ if
\begin{equation}
\nu_iQ_{ij}=\nu_jQ_{ji}\qquad\text{for all }i,j.
\label{eq:db}
\end{equation}
\end{definition}

Two distinct senses of reversibility appear in this paper and we keep them apart
throughout. A chain is \emph{reversible} when it satisfies the detailed balance
\eqref{eq:db}, which is a symmetry of its stationary law, not of its dynamics. Its
time evolution is nonetheless \emph{dissipative}, or dynamically irreversible. The
semigroup $T_t$ contracts toward equilibrium, the relative entropy decreases, and
one cannot run $T_t$ backwards in any stable way. The detailed balance \eqref{eq:db}
is exactly the property that will let us store this dissipative evolution of a
detailed-balance reversible chain inside a deterministic flow that is reversible in
the dynamical sense, namely a one-parameter group that can be run backwards.

\subsection{Uniformisation}
\label{sub:setup-unif}

Fix a uniformisation rate
\begin{equation}
\Lambda\ge\max_i(-Q_{ii}),\qquad \Lambda>0,
\label{eq:Lambda}
\end{equation}
and set
\begin{equation}
P:=\Id+\tfrac1\Lambda Q.
\label{eq:P}
\end{equation}
Because $\Lambda$ dominates every exit rate, $P$ has nonnegative entries and rows
summing to $1$, so it is row-stochastic. It is the jump kernel of the chain watched
at the ticks of an independent Poisson clock of rate $\Lambda$, which is the content
of uniformisation. Two facts follow at once. First, since
$Q^{T}=\Lambda(P^{T}-\Id)$,
\begin{equation}
T_t=e^{tQ^{T}}=e^{\Lambda t(P^{T}-\Id)}.
\label{eq:Tt-P}
\end{equation}

Second, dividing \eqref{eq:db} by $\Lambda$ off the diagonal shows that $P$ is
reversible with respect to $\nu$, that is $\nu_iP_{ij}=\nu_jP_{ji}$. Expanding the
exponential in \eqref{eq:Tt-P} gives the standard Poisson representation
\begin{equation}
T_t=e^{-s}\sum_{k\ge0}\frac{s^{k}}{k!}\,(P^{T})^{k},\qquad s:=\Lambda t,
\label{eq:poisson-rep}
\end{equation}
the average of the $k$-step kernels over the Poisson$(s)$ number of clock ticks. We
will meet \eqref{eq:poisson-rep} again in \cref{sec:dilation}, where the doubled
state space splits this sum into its even and odd halves.

\subsection{The square-root (Hermitian) gauge}
\label{sub:setup-gauge}

Reversibility lets us symmetrise $P$ by a diagonal change of basis. Let
\begin{equation}
D:=\diag\bigl(\sqrt{\nu_1},\dots,\sqrt{\nu_M}\bigr),
\label{eq:D}
\end{equation}
a positive diagonal matrix, and define the square-root gauge
\begin{equation}
A:=DPD^{-1}.
\label{eq:A}
\end{equation}

\begin{lemma}[The gauge is symmetric exactly under detailed balance]
\label{lem:gauge-sym}
$P$ is reversible with respect to $\nu$ if and only if $A=DPD^{-1}$ is symmetric.
When it holds,
\begin{equation}
A=DPD^{-1}=D^{-1}P^{T}D,\qquad A^{T}=A.
\label{eq:A-two-forms}
\end{equation}
\end{lemma}

\begin{proof}
By definition $A_{ij}=\sqrt{\nu_i/\nu_j}\,P_{ij}$. Hence
\[
A_{ij}=A_{ji}
\iff \sqrt{\tfrac{\nu_i}{\nu_j}}\,P_{ij}=\sqrt{\tfrac{\nu_j}{\nu_i}}\,P_{ji}
\iff \nu_iP_{ij}=\nu_jP_{ji},
\]
which is detailed balance. When detailed balance holds, its entries satisfy
\[
(D^{-1}P^{T}D)_{ij}
=\sqrt{\frac{\nu_j}{\nu_i}}\,P_{ji}
=\sqrt{\frac{\nu_i}{\nu_j}}\,P_{ij}
=A_{ij},
\]
which gives \eqref{eq:A-two-forms}.
\end{proof}

\begin{lemma}[Spectrum of the gauge]
\label{lem:gauge-spec}
$A$ is real symmetric and similar to the stochastic matrix $P$. Its eigenvalues are
therefore real, equal to those of $P$, and contained in $[-1,1]$. The eigenvalue
$1$ is attained by the unit vector $\sqrt\nu:=D\mathbf 1$, whose entries are
$\sqrt{\nu_i}$.
\end{lemma}

\begin{proof}
$A=DPD^{-1}$ is a similarity, so $\spec(A)=\spec(P)$, and symmetry comes from
\cref{lem:gauge-sym}. A stochastic matrix has spectral radius $1$, because its rows
sum to $1$, so $\mathbf 1$ is a right eigenvector with eigenvalue $1$, and
$\|P\|_\infty=1$ bounds every eigenvalue by $1$ in modulus. Since $A$ is symmetric
its eigenvalues are real, hence they lie in $[-1,1]$. Finally
$A\sqrt\nu=DPD^{-1}D\mathbf 1=DP\mathbf 1=D\mathbf 1=\sqrt\nu$, and
$\|\sqrt\nu\|^{2}=\sum_i\nu_i=1$.
\end{proof}

The point of the gauge is that the entire Markov semigroup is now carried by a
symmetric matrix. Writing $s=\Lambda t$ for the rescaled time and combining
\eqref{eq:Tt-P} with \eqref{eq:A}, we obtain
\begin{equation}
T_t=D\,e^{s(A-\Id)}D^{-1},\qquad s=\Lambda t.
\label{eq:gauged-semigroup}
\end{equation}
Indeed $D^{-1}Q^{T}D=\Lambda(D^{-1}P^{T}D-\Id)=\Lambda(A-\Id)$ by
\eqref{eq:A-two-forms}, and exponentiating $t$ times this conjugated generator gives
\eqref{eq:gauged-semigroup}. The damping factor $e^{-s}$ of the real slice is
already visible, as the $-\Id$ inside $e^{s(A-\Id)}=e^{-s}e^{sA}$.

\begin{remark}[The gauged generator]
\label{rem:gauged-gen}
It is worth isolating the combination that appears in
\eqref{eq:gauged-semigroup}. The \emph{gauged generator}
\begin{equation}
\gQ:=D^{-1}Q^{T}D=\Lambda(A-\Id),\qquad \gQ_{ij}=\sqrt{\nu_i/\nu_j}\;Q_{ij}\ (i\ne j),
\label{eq:gauged-gen}
\end{equation}
is symmetric under detailed balance by the same computation as
\cref{lem:gauge-sym}. On the finite state space considered here, $\gQ$ and $A$
carry the same information through \eqref{eq:gauged-gen}, and the uniformised
form $A$ is the one used throughout the dilation below.
\end{remark}

\subsection{A running instance: the one-ball Ehrenfest chain}
\label{sub:setup-twostate}

\begin{example}[The one-ball symmetric Ehrenfest chain]
\label{ex:two-state}
Let $M=2$ with symmetric rates
$Q=\left(\begin{smallmatrix}-1&1\\ 1&-1\end{smallmatrix}\right)$, so each state
leaves at rate $1$. Then $\nu=(\tfrac12,\tfrac12)$, detailed balance is automatic,
and the choice $\Lambda=1$ gives the pure flip
$P=\Id+Q=\left(\begin{smallmatrix}0&1\\ 1&0\end{smallmatrix}\right)$. Here
$D=\tfrac1{\sqrt2}\Id$, so
$A=P=\left(\begin{smallmatrix}0&1\\ 1&0\end{smallmatrix}\right)$
with $\spec(A)=\{+1,-1\}$. The semigroup is
\[
T_t=e^{t(A-\Id)}
=e^{-t}\begin{pmatrix}\cosh t&\sinh t\\ \sinh t&\cosh t\end{pmatrix}
=\frac12\begin{pmatrix}1+e^{-2t}&1-e^{-2t}\\ 1-e^{-2t}&1+e^{-2t}\end{pmatrix}.
\]
Started in state $1$, the probability of being in state $1$ at time $t$ is
$p(t)=\tfrac12(1+e^{-2t})$. This single function, the relaxation of one two-state
chain, reappears as the single-ball factor behind the entire Ehrenfest analysis in
\cref{sec:ehrenfest}. The relaxation rate is the eigenvalue $-2$ of $Q$,
equivalently the factor $e^{-2t}$, and we will meet it again as the quantum clock
through $\cos\theta(t)=e^{-2t}$.
\end{example}

% ======================================================================

\section{The doubled complex-orthogonal dilation}
\label{sec:dilation}

We now build the central object, an entire complex-orthogonal linear group $W_z=e^{zK}$ whose
real slice carries the Markov semigroup and whose imaginary slice is unitary. The
construction starts from the symmetric gauge $A$ of \cref{sec:setup} and uses two
moves, a doubling of the state space and a diagonal unitary twist.

These moves remove the main obstacle. The gauged semigroup
\eqref{eq:gauged-semigroup} is driven by $e^{sA}$ with $A$ symmetric. A symmetric
generator stretches and contracts along its eigenvectors, since
$e^{sA}u_j=e^{s\alpha_j}u_j$. This is the opposite of a rotation, and one cannot
read $e^{sA}$ as orthogonal, because orthogonal matrices have unimodular
eigenvalues while $e^{sA}$ does not. The doubling supplies the extra dimensions in
which each stretching direction can pair with a partner and become a two-dimensional hyperbolic
rotation (a boost), and the twist aligns those hyperbolic rotations into a single
complex-orthogonal group. The word \emph{dilation} is used in the spirit of
operator theory, where a contraction semigroup is realised inside a group on a
larger space. The classical unitary theory is Sz.-Nagy and Foia\c{s}
\cite{SzNagy}. The dilation below is finite-dimensional, exact, and lands in the
complex orthogonal group rather than in a unitary group.

\subsection{The doubled chain and the parity of the jump count}
\label{sub:dil-double}

Double the state space to $\hOm=\Omega\times\{+,-\}$, with vectors
$x=\binom{x^{+}}{x^{-}}\in\RR^{2M}$. Using the Kronecker product of
\cref{sub:primer-pauli}, define the doubled stochastic matrix, generator, and
semigroup
\begin{equation}
\hP=\begin{pmatrix}0&P\\ P&0\end{pmatrix}=\sigmax\otimes P,\qquad
\hQ=\Lambda(\hP-\Id_{2M}),\qquad \hT_t=e^{t\hQ^{T}}.
\label{eq:doubled-markov}
\end{equation}
This is an honest continuous-time Markov chain on $2M$ states. At each Poisson tick
it performs one $P$-move on $\Omega$ and flips the sheet label $\pm$. The flip
gives the doubled chain a transparent probabilistic meaning. Throughout, a
\emph{jump} of the uniformised chain means a tick of its Poisson clock. When
$P_{ii}>0$ some ticks leave the register unchanged, and such jumps
also count toward the parity below.

\begin{lemma}[The sheet is the jump-count parity]
\label{lem:parity}
Run the uniformised chain of \cref{sec:setup}, a Poisson$(\Lambda)$ clock with
$P$-moves at the ticks, and start the doubled chain at $(i,+)$. At time $t$ the
doubled chain is at $(j,+)$ if the original chain is at $j$ after an even number of
ticks, and at $(j,-)$ if after an odd number. Consequently, for the embedding
$\hpi_0=\binom{\pi_0}{0}$,
\begin{equation}
\hT_t\binom{\pi_0}{0}
=\binom{\,e^{-s}\cosh(sP^{T})\,\pi_0\,}{\,e^{-s}\sinh(sP^{T})\,\pi_0\,},
\qquad s=\Lambda t,
\label{eq:parity-split}
\end{equation}
so the two sheets carry the even and odd halves of the Poisson representation
\eqref{eq:poisson-rep}, and their sum returns $T_t\pi_0$.
\end{lemma}

\begin{proof}
The first statement is immediate, because every tick flips the sheet, so after $k$
ticks the sheet is $+$ exactly when $k$ is even, and the register marginal is the
original uniformised chain since the register move at each tick is an independent
$P$-move. For \eqref{eq:parity-split}, write $\hT_t=e^{-s}e^{s\hP^{T}}$ and use
$\hP^{T}=\sigmax\otimes P^{T}$ with $(\sigmax)^{2}=\Id_2$. Even powers of $\hP^{T}$
equal $\Id_2\otimes(P^{T})^{2j}$ and odd powers equal
$\sigmax\otimes(P^{T})^{2j+1}$, so the exponential series splits into
$\cosh(sP^{T})$ on the diagonal blocks and $\sinh(sP^{T})$ on the antidiagonal
blocks. Applying this to $\binom{\pi_0}{0}$ gives \eqref{eq:parity-split}, and
$\cosh x+\sinh x=e^{x}$ recovers \eqref{eq:poisson-rep}.
\end{proof}

The doubling is a bookkeeping that lets the new chain remember the parity of
its jump count, and this lets a group, which is invertible
and deterministic, carry a semigroup, which is contractive and stochastic.

\subsection{The twist and the generator \texorpdfstring{$K$}{K}}
\label{sub:dil-twist}

The gauge of the doubled chain, using $\hD:=\diag(D,D)=\Id_2\otimes D$, is the
symmetric block-antidiagonal matrix
\begin{equation}
\hL:=\hD\,\hP\,\hD^{-1}=\begin{pmatrix}0&A\\ A&0\end{pmatrix}=\sigmax\otimes A.
\label{eq:hL}
\end{equation}
On each eigenplane of $A$ this is the boost generator of \cref{sub:primer-boost},
and $e^{s\hL}$ is the hyperbolic rotation $\beta_{s\alpha}$ on that plane. Now
introduce the unitary \emph{twist}
\begin{equation}
V:=\diag(\Id_M,\,i\Id_M)=v\otimes\Id_M,\qquad v=\diag(1,i),
\label{eq:V}
\end{equation}
and define the generator central to the whole paper,
\begin{equation}
K:=V\hL V^{-1}=(v\sigmax v^{-1})\otimes A=\sigmay\otimes A
=\begin{pmatrix}0&-iA\\ iA&0\end{pmatrix}.
\label{eq:K}
\end{equation}

The twist rotated the ancilla $\sigmax$ into $\sigmay$ (\cref{sub:primer-pauli}),
converting the hyperbolic generator $\hL$ into $K$. Note also the identity
\begin{equation}
K^{2}=(\sigmay)^{2}\otimes A^{2}=\Id_2\otimes A^{2},
\label{eq:Ksq}
\end{equation}
which we use repeatedly. The whole pipeline reads
\begin{equation}
Q \;\xrightarrow{\ \text{uniformise}\ }\; P \;\xrightarrow{\ \text{gauge}\ }\; A
\;\xrightarrow{\ \text{double and twist}\ }\; K=\sigmay\otimes A
\;\xrightarrow{\ \exp\ }\; W_z=e^{zK}.
\label{eq:pipeline}
\end{equation}

\subsection{The entire complex-orthogonal group}
\label{sub:dil-group}

We now state the central algebraic fact. Recall from \cref{sub:primer-pairings} the
complex-orthogonal group $\OO(2M,\CC)$, which preserves the bilinear pairing
$\Phi^{T}\Psi$, and the unitary group $\UU(2M)$, which preserves the Hermitian
pairing $\Phi^{*}\Psi$. These are different groups, and the difference is the whole
point of using the prefix ``pseudo''.

\begin{theorem}[Complex-orthogonal dilation]
\label{thm:dilation}
The generator $K=\sigmay\otimes A$ of \eqref{eq:K} satisfies
\begin{equation}
K^{T}=-K\qquad\text{and}\qquad K^{*}=K,
\label{eq:K-double}
\end{equation}
that is, $K$ is antisymmetric for the transpose and Hermitian for the conjugate
transpose. For $z\in\CC$ set $W_z:=e^{zK}$. Then $z\mapsto W_z$ is an entire
homomorphism from $(\CC,+)$ into $\OO(2M,\CC)$, with
\[
W_{z+w}=W_zW_w,\qquad W_0=\Id,\qquad W_z^{T}W_z=\Id,
\]
and it has the closed block form
\begin{equation}
W_z=\begin{pmatrix}\cosh(zA)&-i\sinh(zA)\\ i\sinh(zA)&\cosh(zA)\end{pmatrix}.
\label{eq:Wz}
\end{equation}
\end{theorem}

\begin{proof}
By the Kronecker rules of \cref{sub:primer-pauli},
$K^{T}=(\sigmay)^{T}\otimes A^{T}=(-\sigmay)\otimes A=-K$, using $A^{T}=A$. Because
$\sigmay$ is purely imaginary and $A$ is real, $K$ is purely imaginary, so
$K^{*}=\overline{K^{T}}=\overline{-K}=K$. The series $e^{zK}$ is entire in $z$, and
the homomorphism law holds because $K$ commutes with itself. For orthogonality,
\[
W_z^{T}W_z=e^{zK^{T}}e^{zK}=e^{-zK}e^{zK}=\Id .
\]
For the block form, \eqref{eq:Ksq} shows that even powers of $K$ are
$\Id_2\otimes A^{2j}$ and odd powers are $\sigmay\otimes A^{2j+1}$. Summing the
exponential series, the even part assembles $\cosh(zA)$ on the diagonal blocks and
the odd part assembles $\sinh(zA)$ against $\sigmay$, which gives \eqref{eq:Wz}.
\end{proof}

\begin{remark}[Boosts on eigenplanes]
\label{rem:boost-planes}
Let $u_j$ be an orthonormal eigenvector of $A$ with eigenvalue $\alpha_j$. The
two-dimensional space spanned by $\binom{u_j}{0}$ and $\binom{0}{u_j}$ is invariant
under $W_z$, and on it \eqref{eq:Wz} restricts to the two-by-two matrix
\[
\begin{pmatrix}\cosh(z\alpha_j)&-i\sinh(z\alpha_j)\\
i\sinh(z\alpha_j)&\cosh(z\alpha_j)\end{pmatrix},
\]
which at real $z=s$ is conjugate, by the fixed twist $v=\diag(1,i)$, to the Lorentz
boost $\beta_{s\alpha_j}$ of \cref{sub:primer-boost} with rapidity $s\alpha_j$, and
at imaginary $z=i\theta$ is a rotation by the angle $\theta\alpha_j$. The whole
flow is an orthogonal direct sum of these two-dimensional blocks, one per
eigenvalue of the gauge. This block picture is the engine of the Ehrenfest analysis
in \cref{sec:ehrenfest}.
\end{remark}

\begin{remark}[Why ``NUO'': orthogonal but not unitary]
\label{rem:nuo}
For real $s\neq0$ the matrix $W_s\in\OO(2M,\CC)$ preserves the bilinear form
$\Phi^{T}\Psi$, but it is not unitary, because $K^{*}=K$ gives
\[
W_s^{*}W_s=e^{sK^{*}}e^{sK}=e^{2sK}\ne\Id .
\]
So $W_s$ is non-unitary orthogonal. It conserves the complex bilinear length
$\Phi^{T}\Phi$ while distorting the Hermitian length $\Phi^{*}\Phi$, and that
Hermitian distortion $e^{2sK}$ is exactly the room in which the chain's dissipation
lives. On the imaginary slice, by contrast, $K^{*}=K$ makes $W_{i\theta}$
orthogonal and unitary at once (\cref{sec:imaginary}), and that coincidence is what
turns the imaginary slice into honest quantum mechanics.
\end{remark}

The next section shows that the real slice $W_s$ contains the Markov chain.

% ======================================================================

\section{Decoding: recovering the chain from the real slice}
\label{sec:decode}

The complex-orthogonal dilation would be just a curiosity if the chain could not be read
back out of it. We now prove that it can, exactly and through a simple recipe, and
we identify the probabilistic meaning of each step. This is the precise sense in
which the deterministic complex-orthogonal flow represents the dissipative chain
with no loss.

Three fixed ingredients turn $W_s$ back into a probability law. First, the
\emph{column-gauge matrix} $\Bc:=V\hD^{-1}$, with $V$ and $\hD$ from
\cref{sec:dilation}. Second, the \emph{marginal projection}
\begin{equation}
\Gamma\binom{x^{+}}{x^{-}}:=x^{+}+x^{-}\in\RR^{M},
\label{eq:Gamma}
\end{equation}
which adds the two sheets. Third, the \emph{embedding} of an initial law
$\pi_0\in\RR^{M}$ as $\hpi_0:=\binom{\pi_0}{0}$, placing all mass on the $+$ sheet
at time $0$.

\begin{theorem}[Decoding and reconstruction]
\label{thm:decode}
Let $Q$ be a finite irreducible reversible generator, with all objects defined from
$Q,\nu,\Lambda$ as in \cref{sec:setup,sec:dilation}. For $s=\Lambda t$,
\begin{equation}
\hT_t=e^{t\hQ^{T}}=e^{-s}\,\Bc^{-1}W_s\Bc,
\label{eq:doubled-decode}
\end{equation}
and consequently, for every initial probability law $\pi_0$,
\begin{equation}
\pi_t=T_t\pi_0=\Gamma\!\bigl(e^{-s}\,\Bc^{-1}W_s\Bc\,\hpi_0\bigr),\qquad
s=\Lambda t.
\label{eq:decode}
\end{equation}
\end{theorem}

\begin{proof}
The doubled gauge conjugates $\hQ^{T}$ to a multiple of $\hL-\Id$, since
$\hD^{-1}\hQ^{T}\hD=\Lambda(\hL-\Id)$ is the blockwise version of
\eqref{eq:gauged-semigroup}. Hence
\[
\hT_t=\hD\,e^{s(\hL-\Id)}\hD^{-1}=e^{-s}\,\hD\,e^{s\hL}\hD^{-1}.
\]
By \eqref{eq:K}, $K=V\hL V^{-1}$, so $e^{s\hL}=V^{-1}W_sV$. Because $V$ acts on the
ancilla and $\hD$ on the register, the two commute, and
\[
\hT_t=e^{-s}\hD V^{-1}W_sV\hD^{-1}
=e^{-s}(V\hD^{-1})^{-1}W_s(V\hD^{-1})=e^{-s}\Bc^{-1}W_s\Bc,
\]
which is \eqref{eq:doubled-decode}. For \eqref{eq:decode}, note that the twist
cancels once the conjugation by $\Bc$ is carried out, since
$\Bc^{-1}W_s\Bc=\hD\,e^{s\hL}\hD^{-1}$ and
$e^{s\hL}=\left(\begin{smallmatrix}\cosh(sA)&\sinh(sA)\\
\sinh(sA)&\cosh(sA)\end{smallmatrix}\right)$ by the same even and odd splitting
that produced \eqref{eq:Wz}. Applying this to the embedded law,
\begin{equation}
\hT_t\binom{\pi_0}{0}
=e^{-s}\binom{D\cosh(sA)D^{-1}\pi_0}{\,D\sinh(sA)D^{-1}\pi_0\,},
\label{eq:sheets}
\end{equation}
a real vector, and adding the two sheets. Using $\cosh(sA)+\sinh(sA)=e^{sA}$,
\[
\Gamma\bigl(\hT_t\hpi_0\bigr)=e^{-s}D\,e^{sA}D^{-1}\pi_0
=D\,e^{s(A-\Id)}D^{-1}\pi_0=T_t\pi_0
\]
by \eqref{eq:gauged-semigroup}.
\end{proof}

\begin{corollary}[The probabilistic meaning of each step]
\label{cor:decode-meaning}
Comparing the sheets \eqref{eq:sheets} with the parity split
\eqref{eq:parity-split} of \cref{lem:parity}, the two gauged sheets are exactly the
even-jump and odd-jump halves of the uniformised chain. In the decoding
\eqref{eq:decode} the similarity $\Bc$ removes the square-root gauge and the twist,
the damping $e^{-s}$ is the Poisson normalisation $e^{-\Lambda t}$ of the tick
count, and the marginal $\Gamma$ sums out the parity bit.
\end{corollary}

The decoding is therefore the composition of three simple operations, the
similarity $\Bc$, the scalar $e^{-s}$, and the linear marginal $\Gamma$, with one
entire group element $W_s$. The only non-injective step is the final marginal sum,
and it is balanced by the embedding, which started all mass on one sheet.

\begin{example}[Lossless reconstruction of the symmetric two-state chain]
\label{ex:two-state-decode}
Take the chain of \cref{ex:two-state}, with $M=2$, $\Lambda=1$, $A=\sigmax$, and
$D=\tfrac1{\sqrt2}\Id_2$. Since $A^{2}=\Id_2$, we have
$\cosh(sA)=\cosh s\,\Id_2$ and $\sinh(sA)=\sinh s\,\sigmax$, so the real slice is
the explicit matrix
\[
W_s=\begin{pmatrix}\cosh s\,\Id_2 & -i\sinh s\,\sigmax\\[2pt]
i\sinh s\,\sigmax & \cosh s\,\Id_2\end{pmatrix}.
\]
Here $\hD=\tfrac1{\sqrt2}\Id_4$, so $\Bc=V\hD^{-1}=\sqrt2\,\diag(\Id_2,i\Id_2)$.
Start in state $1$, $\pi_0=(1,0)^{T}$, $\hpi_0=(1,0,0,0)^{T}$. The pseudo-wave
$\Phi_0=\Bc\hpi_0=\sqrt2\,(1,0,0,0)^{T}$ evolves to
\[
\Phi_s=W_s\Phi_0=\sqrt2\,(\cosh s,\ 0,\ 0,\ i\sinh s)^{T},
\]
a complex vector whose bilinear length
$\Phi_s^{T}\Phi_s=2(\cosh^{2}s-\sinh^{2}s)=2$ is conserved, while its Hermitian
length $\Phi_s^{*}\Phi_s=2(\cosh^{2}s+\sinh^{2}s)$ grows. Decoding applies the
gauge and the damping,
\[
e^{-s}\Bc^{-1}\Phi_s=e^{-s}(\cosh s,\ 0,\ 0,\ \sinh s)^{T},
\]
whose sheets $x^{+}=e^{-s}(\cosh s,0)$ and $x^{-}=e^{-s}(0,\sinh s)$ are the
even-jump and odd-jump halves. Their marginal sum is
\[
\Gamma\bigl(e^{-s}\Bc^{-1}\Phi_s\bigr)=x^{+}+x^{-}=e^{-s}(\cosh s,\ \sinh s)
=\Bigl(\tfrac{1+e^{-2s}}2,\ \tfrac{1-e^{-2s}}2\Bigr)=\pi_t,
\]
exactly the law of \cref{ex:two-state}, with $s=t$ because $\Lambda=1$. The
dissipative two-state relaxation has
been reproduced by a deterministic rotation in $\CC^{4}$ followed by a scalar and a
sum. This is precisely the $n=1$ instance of the symmetric Ehrenfest urn developed
in \cref{sec:ehrenfest}.
\end{example}

The example makes precise where the arrow of time lives. It is not in the flow
$W_s$, which is invertible, conserves a bilinear length, and can be run backwards.
It is in the decoding read-out, the gauge $\Bc$, the damping $e^{-s}$ and the many-to-one
marginal $\Gamma$. Strip those away and what remains is a rigid complex rotation.
The next two sections study that rotation in its own right, first through its
quadratic pseudo-density (\cref{sec:pseudostate}) and then through its unitary form
on the imaginary axis (\cref{sec:imaginary}).

% ======================================================================

\section{The pseudo-density, its bilinear von Neumann equation, and the pseudo-Bloch vector}
\label{sec:pseudostate}

The dilation gives, for free, two further objects that look exactly like the state
vector and the density matrix of quantum mechanics, with the essential change noted
in the introduction. The pairing is the bilinear $\Phi^{T}\Psi$ rather than the
Hermitian $\Phi^{*}\Psi$. The prefix ``pseudo'' highlights this change. The genuine
density matrix, von Neumann equation, and Bloch vector are standard objects in quantum mechanics
\cite{Sakurai,BengtssonZyczkowski}, while the bilinear,
complex-symmetric variants below are part of the present construction.

\subsection{The pseudo-wave and the pseudo-density}
\label{sub:pseudo-def}

\begin{definition}[Pseudo-wave and pseudo-density]
\label{def:pseudo}
For an initial law $\pi_0$, $\hpi_0:=\binom{\pi_0}{0}$, set $\Phi_0:=\Bc\hpi_0$ and, for $z\in\CC$,
\[
\Phi_z:=W_z\Phi_0\in\CC^{2M}\quad(\text{the \emph{pseudo-wave}}),\qquad
\rho_z:=\Phi_z\Phi_z^{T}\in\Sym_{2M}(\CC)\quad(\text{the \emph{pseudo-density}}).
\]
The superscript $T$ is the plain transpose, with no complex conjugation, so
$\rho_z$ is complex symmetric and in general not Hermitian.
\end{definition}

\begin{theorem}[Pseudo-Schr\"odinger and bilinear von Neumann equations]
\label{thm:bilinear}
The pseudo-wave and pseudo-density obey
\begin{equation}
\dot\Phi_z=K\Phi_z,\qquad \dot\rho_z=[K,\rho_z]=K\rho_z-\rho_zK.
\label{eq:bilinear-vn}
\end{equation}
Moreover, for all $z$,
\begin{equation}
\rho_z^{T}=\rho_z,\qquad
\Tr\rho_z=\Phi_z^{T}\Phi_z=\Phi_0^{T}\Phi_0,
\qquad
\rho_z^{2}=(\Tr\rho_z)\,\rho_z.
\label{eq:rho-identities}
\end{equation}
\end{theorem}

\begin{proof}
The wave equation is the derivative of $W_z=e^{zK}$. For the density,
\[
\dot\rho_z=(K\Phi_z)\Phi_z^{T}+\Phi_z(K\Phi_z)^{T}
=K\rho_z+\Phi_z\Phi_z^{T}K^{T}=K\rho_z-\rho_zK,
\]
where the last step uses $K^{T}=-K$ from \eqref{eq:K-double}. Symmetry of $\rho_z$
is immediate. The trace is conserved because $W_z$ is complex-orthogonal,
\[
\Tr\rho_z=\Tr(\Phi_z\Phi_z^{T})=\Phi_z^{T}\Phi_z
=\Phi_0^{T}W_z^{T}W_z\Phi_0=\Phi_0^{T}\Phi_0.
\]
Finally $\rho_z^{2}=\Phi_z(\Phi_z^{T}\Phi_z)\Phi_z^{T}=(\Tr\rho_z)\rho_z$.
\end{proof}

The first equation of \eqref{eq:bilinear-vn} is a pseudo-Schr\"odinger equation,
and the second has the exact form of the quantum von Neumann equation
$\dot\sigma=-i[H,\sigma]$ of \eqref{eq:vn-quantum}, with $K$ in the role of $-iH$.
This is not just an analogy, since the generator $K$ is Hermitian, $K^{*}=K$, so
$-iK$ is a legitimate skew-Hermitian operator, and on the imaginary slice it will
generate honest unitary dynamics (\cref{sec:imaginary}).

\begin{remark}[The pseudo-density is not a quantum density matrix]
\label{rem:false-friend-density}
In quantum mechanics the density is the Hermitian outer product
$\sigma=\psi\psi^{*}$, with $\sigma=\sigma^{*}\geq 0$ and $\Tr\sigma=1$. Here
$\rho_z=\Phi_z\Phi_z^{T}$ uses the transpose, so $\rho_z$ is complex symmetric, in
general neither Hermitian nor positive, and its trace is a complex constant. For
the physical embedding $\hpi_0:=\binom{\pi_0}{0}$ one computes
$\Phi_0^{T}\Phi_0=\sum_i\pi_{0,i}^{2}/\nu_i>0$, a $\chi^{2}$-type quantity equal to
$1+\chi^{2}(\pi_0\,\|\,\nu)$, but for general complex initial data it can be any
complex number, including zero. A genuine Hermitian density appears only on the imaginary
slice.
\end{remark}

\begin{remark}[Why the bilinear pairing, and why it is holomorphic]
\label{rem:holomorphic}
Nothing in the proof used that $z$ is real. The map $z\mapsto\rho_z$ is an entire
$\Sym_{2M}(\CC)$-valued solution of $\tfrac{d}{dz}\rho_z=[K,\rho_z]$, with
$\Phi_z^{T}\Phi_z$ and $\Tr\rho_z$ constant and $\rho_z^{2}=(\Tr\rho_z)\rho_z$
throughout. The real and imaginary slices are two lines in the domain of this one
holomorphic family. The bilinear pairing is exactly what makes the family
holomorphic, because the Hermitian outer product $\Phi_z\Phi_z^{*}$ depends on
$\bar z$ and is not holomorphic.
\end{remark}

\subsection{The pseudo-Bloch vector equation}
\label{sub:pseudo-bloch}

In quantum mechanics one expands a density matrix in a basis of traceless Hermitian
matrices (like the generalised Gell-Mann basis, which, for one qubit, are the Pauli
matrices), and obtains the Bloch vector. Our bilinear analogue requires only the
symmetric part of such a basis.

\begin{proposition}[Pseudo-Bloch equation]
\label{prop:pseudo-bloch}
Let $N=2M$ and let $\{\Sigma_a\}_{a=1}^{d}$, with $d=\tfrac{N(N+1)}2-1$, be a
trace-orthonormal basis of the traceless complex-symmetric matrices, so that
$\Tr(\Sigma_a\Sigma_b)=\delta_{ab}$, $\Tr\Sigma_a=0$, and
$\Sigma_a^{T}=\Sigma_a$. Orthonormality here is meant for the bilinear trace form
$\Tr(\Sigma_a\Sigma_b)$, with no complex conjugation, and such a basis exists
because that form is nondegenerate on the complex-symmetric matrices, for instance
the symmetrised units $(E_{ij}+E_{ji})/\sqrt2$ for $i<j$ together with an
orthonormal set of $N-1$ real traceless diagonal matrices work. The
$\Sigma_a$ are complex symmetric, not Hermitian, so they are the bilinear analogue
of the Gell-Mann matrices rather than the matrices themselves. Write
\[
\rho_z=\frac{\Tr\rho_z}{N}\,\Id+\sum_{a=1}^{d}r_a(z)\,\Sigma_a,\qquad
r_a(z)=\Tr(\rho_z\Sigma_a).
\]
Then the coordinate vector obeys the linear flow
\[
\dot r_a(z)=\sum_{b}\Omega_{ab}\,r_b(z),\qquad
\Omega_{ab}=\Tr\bigl(\Sigma_a[K,\Sigma_b]\bigr),\qquad \Omega^{T}=-\Omega.
\]
Hence the flow lies in $\OO(d,\CC)$ and preserves the bilinear quantity
$\sum_a r_a^{2}$.
\end{proposition}

\begin{proof}
Since $K^{T}=-K$ and $\Sigma_b^{T}=\Sigma_b$, the commutator $[K,\Sigma_b]$ is
again symmetric and traceless, so the expansion is consistent. Differentiating
$r_a=\Tr(\rho_z\Sigma_a)$ and inserting $\dot\rho_z=[K,\rho_z]$ gives the indicated
linear system, using the cyclicity of the trace to move the commutator onto
$\Sigma_a$. For antisymmetry, with $\Sigma^{T}=\Sigma$, $K^{T}=-K$, and
$\Tr(X)=\Tr(X^{T})$ on each term,
\[
\Omega_{ab}=\Tr(\Sigma_aK\Sigma_b)-\Tr(K\Sigma_a\Sigma_b)
=-\Tr(\Sigma_bK\Sigma_a)+\Tr(K\Sigma_b\Sigma_a)=-\Omega_{ba}.
\]
An antisymmetric generator exponentiates into $\OO(d,\CC)$, which preserves
$\sum_a r_a^{2}$.
\end{proof}

\begin{remark}[A non-compact pseudo-Bloch sphere]
\label{rem:noncompact-bloch}
For a rank-one pseudo-density $\rho=\Phi\Phi^{T}$ the conserved quantity is a
quadric, \emph{not} a sphere,
\begin{equation}
\sum_{a=1}^{d}r_a^{2}=\frac{N-1}{N}\,(\Tr\rho)^{2}.
\label{eq:quadric}
\end{equation}
Indeed $\Tr(\rho^{2})=(\Tr\rho)^{2}$ by \eqref{eq:rho-identities}, while the
orthonormal expansion gives
$\Tr(\rho^{2})=\tfrac{(\Tr\rho)^{2}}N+\sum_a r_a^{2}$, and subtracting gives
\eqref{eq:quadric}. Because the $r_a$ are complex and the form $\sum_a r_a^{2}$
carries no conjugation, the pure-state locus is a complex affine quadric, a
non-compact surface, not the compact Bloch sphere of quantum mechanics. The
pseudo-Bloch orbits are $\OO(d,\CC)$ orbits on this quadric. Only on the imaginary
slice may an orbit close up into a genuine compact sphere.
\end{remark}

The pseudo-density is an interesting new object of the construction, a
complex-symmetric ``state'' obeying a commutator flow on a non-compact quadric.
We stress once more that we obtained it for free. Nothing in
\cref{thm:bilinear,prop:pseudo-bloch} required a modelling decision beyond the
choice of the bilinear pairing that made $W_z$ entire in the first place, and
nothing later in this paper uses these equations. They are a structural bonus
whose properties are not yet fully investigated. The next section shows that one
special line in the complex plane, the imaginary axis, removes every ``pseudo''
and returns honest quantum mechanics.

% ======================================================================

\section{Complex time and the unitary imaginary slice}
\label{sec:imaginary}

We have met the real slice $W_s$, which carries the chain, and the bilinear
pseudo-density, the holomorphic object. We now examine the second distinguished
line in the complex plane, the imaginary axis $z=i\theta$, and show that there the
same generator $K$ produces an honest finite-dimensional unitary quantum system.

\subsection{The imaginary slice is unitary quantum mechanics}
\label{sub:imag-unitary}

\begin{theorem}[Imaginary slice as a quantum system]
\label{thm:imag}
Let $N=2M$. For every real $\theta$ the matrix $U_\theta:=W_{i\theta}=e^{i\theta K}$
is real orthogonal, $U_\theta\in\SO(N,\RR)$, and therefore also unitary,
$U_\theta\in\UU(N)$. Equivalently, the imaginary slice is the Schr\"odinger
evolution
\[
U_\theta=e^{-i\theta\HNUO},\qquad \HNUO:=-K=-\sigmay\otimes A=\HNUO^{*},
\]
generated by the Hermitian Hamiltonian $\HNUO$. The genuine quantum density
$\sigma_\theta:=U_\theta\sigma_0U_\theta^{*}$, which is Hermitian, positive, and
unit-trace, obeys the ordinary von Neumann equation
\[
\dfrac{d\sigma_\theta}{d\theta}=-i[\HNUO,\sigma_\theta].
\]
\end{theorem}

\begin{proof}
Set $Y:=-iK$. From \eqref{eq:K-double}, $K^{T}=-K$ and $K$ is purely imaginary,
so
\[
Y^{T}=-iK^{T}=iK=-Y,\qquad \overline{Y}=i\,\overline{K}=i(-K)=Y .
\]
Thus $Y$ is real and antisymmetric. Since
$W_{i\theta}=e^{i\theta K}=e^{-\theta Y}$ is the exponential of a real
antisymmetric matrix, it is real orthogonal with determinant $1$, and a real
orthogonal matrix is unitary. The Schr\"odinger form is the identity
$e^{-i\theta\HNUO}=e^{i\theta K}$, and the density equation is the standard
differentiation of $\sigma_\theta=U_\theta\sigma_0U_\theta^{*}$ for unitary
$U_\theta$.
\end{proof}

Although the Hamiltonian $\HNUO=-\sigmay\otimes A$ has purely imaginary entries,
the evolution it generates is a real rotation of $\RR^{N}\subset\CC^{N}$. The
imaginary slice sits in the intersection
$\OO(N,\CC)\cap\UU(N)=\OO(N,\RR)$, the unique place where the bilinear and
Hermitian geometries agree.

\subsection{Imaginary slice Hamiltonian spectrum}
\label{sub:imag-spectrum}

Let $\{\alpha_j\}_{j=1}^{M}$ be the spectrum of the gauge $A$, with
$\alpha_j\in[-1,1]$ by \cref{lem:gauge-spec} and orthonormal eigenvectors $u_j$.

\begin{proposition}[Energy levels and eigenstates]
\label{prop:imag-spectrum}
Let $\ket{y_\pm}:=\tfrac1{\sqrt2}(1,\pm i)^{T}$ be the eigenvectors of $\sigmay$,
with $\sigmay\ket{y_\pm}=\pm\ket{y_\pm}$. Then the $2M$ orthonormal eigenpairs of
$\HNUO$ are
\[
\HNUO\,\bigl(\ket{y_\pm}\otimes u_j\bigr)=\mp\,\alpha_j\,
\bigl(\ket{y_\pm}\otimes u_j\bigr),\qquad j=1,\dots,M,
\]
so $\spec(\HNUO)=\{\pm\alpha_j\}$, symmetric about $0$. The Markov relaxation modes
and the quantum energy levels are the same list of numbers. The mode with semigroup
decay $e^{s(\alpha_j-1)}$ is the pair of energy levels $\pm\alpha_j$.
\end{proposition}

\begin{proof}
$\HNUO(\ket{y_\pm}\otimes u_j)=-(\sigmay\ket{y_\pm})\otimes(Au_j)
=\mp\alpha_j(\ket{y_\pm}\otimes u_j)$, and the vectors are orthonormal because both
factors are.
\end{proof}

\subsection{Bounded recurrent motion, the converse of the real slice}
\label{sub:imag-recurrence}

The eigenvalues of $U_\theta$ are $e^{\mp i\theta\alpha_j}$, all on the unit
circle, so the orbit of any state is bounded and quasi-periodic. If the $\alpha_j$
are commensurate the motion is exactly periodic, otherwise it fills a torus
densely. This is the exact converse of the real slice $W_s$, whose eigenvalues
$e^{\pm s\alpha_j}$ run off to $0$ and $\infty$. The same $K$ and the same
$\alpha_j$ give an unbounded boost along the real axis and a bounded rotation along
the imaginary axis, which is the $\cosh$ versus $\cos$ dichotomy. Both live in the
one entire flow $W_z$. Along the real axis the modes are boosts and the decoded law
relaxes monotonically. Along the imaginary axis the same modes are rotations and
the state recurs periodically. For the Ehrenfest example below, the clock
$\theta(t)=\arccos(e^{-2t})$, which maps the infinite relaxation half-line
$t\in[0,\infty)$ into the first quarter turn $\theta\in[0,\tfrac\pi2)$, is the
change of variables that lets a periodic motion carry a monotone one without
contradiction. We develop this in \cref{sec:ehrenfest}.

\subsection{The Born read-out}
\label{sub:imag-born}

Unlike the bilinear $\rho_z$ of \cref{sec:pseudostate}, the imaginary-slice
density $\sigma_\theta=U_\theta\sigma_0U_\theta^{*}$ is Hermitian, positive, and
unit-trace for all $\theta$, a bona fide quantum state. Its probabilities are read
by the Born rule. The probability of finding the register in state
$\ket k$ is $\bra k\sigma_\theta\ket k$, a quadratic functional of $U_\theta$. For
the Ehrenfest chain this quadratic read-out reproduces the classical law $\pi_t$
exactly under the clock $\theta(t)=\arccos(e^{-2t})$, as we show in
\cref{sec:ehrenfest}.

\begin{remark}[The Hamiltonian is the same matrix as the chain, and this is not the chain at imaginary time]
\label{rem:canonicity}
The Hamiltonian $\HNUO=-K$ is the same matrix that generates the real-time chain:
the imaginary slice is not an additional modelling choice but the value of one
entire group at $z=i\theta$, and every finite reversible chain carries a canonical
finite quantum system attached through its symmetric gauge $A$. It is nonetheless
misleading to call $U_\theta$ the Markov semigroup at imaginary time, a naive Wick
rotation $t\mapsto i\theta$ of $T_t$: the accurate statement is that $W_s$ and
$W_{i\theta}$ are two slices of one entire matrix function $W_z$, so the analytic
continuation is exact at the level of the operator, but the recovery of
probabilities is linear in the operator on the real slice, a marginal with
damping, and quadratic on the imaginary slice, the Born rule. Matching the two
outputs forces a nonlinear reparametrisation of time, not the naive $\theta=s$.
\end{remark}

\begin{example}[Imaginary slice of the two-state chain]
\label{ex:two-state-imag}
For the symmetric two-state chain, $A=\sigmax$, $M=2$, $N=4$, and
$\HNUO=-\sigmay\otimes\sigmax$. Since $(\sigmay\otimes\sigmax)^{2}=\Id_4$,
\[
U_\theta=e^{i\theta\,\sigmay\otimes\sigmax}
=\cos\theta\,\Id_4+i\sin\theta\,(\sigmay\otimes\sigmax),
\]
which is real, because $i\,\sigmay\otimes\sigmax$ is real, and orthogonal. The
energy spectrum is $\{+1,+1,-1,-1\}$, so $U_\theta$ is $2\pi$-periodic and the
system precesses on a circle instead of relaxing. This is the same $K$ whose real
slice produced the irreversible decay of \cref{ex:two-state-decode}. The only
change is the direction taken in the complex plane.
\end{example}

% ======================================================================

\section{The Ehrenfest urn in detail}
\label{sec:ehrenfest}

We now illustrate the construction on a single chain, the Ehrenfest urn,
chosen because it is the textbook model of relaxation versus recurrence and because its gauge is
exactly a spin matrix. This section is the paper's single fully worked example. Every object of
\cref{sec:setup,sec:dilation,sec:decode,sec:pseudostate,sec:imaginary}, the gauge,
the spectrum and eigenvectors, the decoding, the boosts, the imaginary slice, the
Born read-out, is computed in closed form, and the smallest instances $n=1$ and
$n=2$ are displayed as explicit matrices. Everything in \cref{sub:ehr-model}
is classical. The PQR enters at \cref{sub:ehr-spin}.

\begin{notation}[Two labels for one basis]
\label{not:ehr-basis}
The register space is $\CC^{n+1}$ and its standard basis is indexed in two ways
throughout this section. The probabilistic label is $k\in\{0,\dots,n\}$, the
number of balls in urn $a$. The spin label is $m=k-\tfrac n2$, the $J_z$
eigenvalue. We write $\ket k$ and $\ket{m}$ for the same vector under the two
labels, so that $\ket{k=n}=\ket{m=+\tfrac n2}$ is the state with every ball in
urn $a$, and $\Jz\ket k=(k-\tfrac n2)\ket k$. Where confusion is possible the
label is written out, as in $\ket{m=k-\tfrac n2}$.
\end{notation}

\subsection{The model and its factorisation over balls}
\label{sub:ehr-model}

The usual symmetric Ehrenfest urn has $n$ indistinguishable balls in two urns
$a$ and $b$. Each ball changes urn at rate $1$, independently of the other balls.
The state $k\in\Omega=\{0,\dots,n\}$ counts the balls in $a$, so $M=n+1$. With
$k$ balls in $a$ the chain moves $k\to k-1$ at rate $k$ and $k\to k+1$ at rate
$n-k$, hence
\begin{equation}
Q_{k,k-1}=k,\qquad Q_{k,k+1}=n-k,\qquad Q_{kk}=-n.
\label{eq:ehr-Q}
\end{equation}

\begin{lemma}[Symmetric binomial stationary law]
\label{lem:ehr-nu}
The chain is reversible with respect to the binomial law
$\nu_k=2^{-n}\binom nk$.
\end{lemma}

\begin{proof}
For a birth-death chain detailed balance reduces to the ladder relation
$\nu_k(n-k)=\nu_{k+1}(k+1)$. This is the binomial identity
$\binom nk(n-k)=\binom n{k+1}(k+1)$, and the weights sum to $1$ by the
binomial theorem.
\end{proof}

The structural fact behind the urn is that the chain is $n$ independent copies of
one two-state chain, watched through the count $k$.

\begin{lemma}[Independence over balls]
\label{lem:ehr-factor}
Let $X_t^{(j)}\in\{a,b\}$ be the urn of ball $j$. The processes
$X_t^{(1)},\dots,X_t^{(n)}$ are independent two-state chains, each jumping $a\to b$
and $b\to a$ at rate $1$, and $k_t=\#\{j:X_t^{(j)}=a\}$.
Starting with all $n$ balls in $a$,
\begin{equation}
\pi_t(k)=\binom nk\,p(t)^{k}\bigl(1-p(t)\bigr)^{n-k},\qquad
p(t)=\frac{1+e^{-2t}}2,
\label{eq:ehr-binom}
\end{equation}
where $p(t)$ is the probability that one ball started in $a$ is in $a$ at time $t$.
This is the function already obtained in \cref{ex:two-state}.
\end{lemma}

\begin{proof}
Each ball moves independently of the others, which is the definition of $n$
independent two-state chains, and the occupancy is their sum. The single-ball law
solves the $2\times2$ forward equation of \cref{ex:two-state}, giving
$p(t)=\tfrac12(1+e^{-2t})$. A sum of $n$ independent indicators with common
success probability $p(t)$ is $\mathrm{Binomial}(n,p(t))$.
\end{proof}

We recover \eqref{eq:ehr-binom} below twice more, spectrally from the Krawtchouk
decomposition (\cref{sub:ehr-spectral}) and from the Born rule of the
imaginary-slice quantum system on the quantum time frame.
The factorisation already explains why the matching clock will not depend on $n$.
The clock tracks a single ball, and in the Majorana picture of
\cref{sub:ehr-connections} each ball is one star.

\subsection{The gauged Ehrenfest generator is a quantum spin}
\label{sub:ehr-spin}

Apply the square-root gauge \eqref{eq:A} to the symmetric Ehrenfest urn with the
minimal uniformisation rate $\Lambda=n$. The diagonal of $A$ vanishes, since
$(n-k)+k=n$, and the off-diagonal entries are
$(\Asym)_{k+1,k}=\tfrac1n\sqrt{(k+1)(n-k)}$. These are exactly the entries of a
spin matrix.

\begin{lemma}[Symmetric Ehrenfest gauge equals a rescaled $J_x$]
\label{lem:Asym-Jx}
For the symmetric Ehrenfest chain with $\Lambda=n$,
\begin{equation}
\Asym=\frac2n\,J_x,
\label{eq:Asym-Jx}
\end{equation}
where $J_x$ is the spin-$\tfrac n2$ matrix of \cref{sub:primer-spin}, indexed by
$k=m+\tfrac n2$.
\end{lemma}

\begin{proof}
Both matrices are real symmetric tridiagonal with zero diagonal, so it suffices to
match one off-diagonal band. With $j=\tfrac n2$ and $k=m+\tfrac n2$, the relevant
entries are
\begin{align*}
(J_x)_{k+1,k}
  &=\frac12\sqrt{(j-m)(j+m+1)}
   =\frac12\sqrt{(n-k)(k+1)},\\
\frac2n(J_x)_{k+1,k}
  &=\frac1n\sqrt{(n-k)(k+1)}
   =(\Asym)_{k+1,k}.
\end{align*}
\end{proof}

\begin{example}[The smallest urns, explicitly]
\label{ex:ehr-small}
For $n=1$ the state space is $\{0,1\}$, the symmetric rates are those of
\cref{ex:two-state}, $\Lambda=1$, and
\[
\Asym=\begin{pmatrix}0&1\\1&0\end{pmatrix}=\sigmax=2J_x^{(1/2)},
\]
so the running two-state example of
\cref{ex:two-state,ex:two-state-decode,ex:two-state-imag} is the $n=1$ Ehrenfest
urn, and its whole PQR was already displayed there as explicit $4\times4$
matrices. For $n=2$ the state space is $\{0,1,2\}$, $\nu=(\tfrac14,\tfrac12,
\tfrac14)$, $\Lambda=2$, and
\[
Q=\begin{pmatrix}-2&2&0\\ 1&-2&1\\ 0&2&-2\end{pmatrix},\qquad
P=\begin{pmatrix}0&1&0\\ \tfrac12&0&\tfrac12\\ 0&1&0\end{pmatrix},\qquad
\Asym=\frac{1}{\sqrt2}\begin{pmatrix}0&1&0\\ 1&0&1\\ 0&1&0\end{pmatrix}
=J_x^{(1)},
\]
the spin-$1$ matrix itself. Its spectrum is $\{-1,0,1\}$ with orthonormal
eigenvectors
\[
u_{-1}=\tfrac12\bigl(1,\,-\sqrt2,\,1\bigr)^{T},\qquad
u_{0}=\tfrac1{\sqrt2}\bigl(1,\,0,\,-1\bigr)^{T},\qquad
u_{+1}=\tfrac12\bigl(1,\,\sqrt2,\,1\bigr)^{T}=\sqrt\nu .
\]
We use these vectors twice below, as Krawtchouk vectors in
\cref{ex:ehr-spectral-n2} and as boost planes in \cref{sub:ehr-connections}.
\end{example}

\begin{corollary}[Symmetric Ehrenfest spectra]
\label{cor:ehr-spec}
$\spec(\Asym)=\{2m/n:m=-\tfrac n2,\dots,\tfrac n2\}$, and the generator satisfies
$\spec(Q)=\{-2j:j=0,\dots,n\}$. The relaxation rates $2j$ are the integer
multiples of the single-ball rate $2$, the $j$-th mode being $j$ balls relaxing
together.
\end{corollary}

\begin{proof}
$\spec(J_x)=\{-\tfrac n2,\dots,\tfrac n2\}$, so $\spec(\Asym)=\tfrac2n\spec(J_x)$.
By \eqref{eq:gauged-gen} the generator $Q$ is similar to $n(\Asym-\Id)$, whose
eigenvalue at $m$ is $n(\tfrac{2m}n-1)=2m-n$. As $m$ runs over
$\{-\tfrac n2,\dots,\tfrac n2\}$ this runs over $\{-2n,-2n+2,\dots,0\}$, which is
$\{-2j:j=0,\dots,n\}$ after the substitution $j=\tfrac n2-m$.
\end{proof}

Feeding \eqref{eq:Asym-Jx} into the generator $K=\sigmay\otimes A$ of \eqref{eq:K},
the Ehrenfest NUO generator is
\begin{equation}
\Ksym=\frac2n\,\sigmay\otimes J_x
\label{eq:Ksym}
\end{equation}
on $\CC^{2}\otimes\CC^{n+1}$. It couples an ancilla spin-$\tfrac12$, which carries
the parity label of \cref{lem:parity} and is acted on by $\sigmay$, to a register
spin-$\tfrac n2$, which carries the urn index and is acted on by $J_x$. Every
question about the Ehrenfest chain has been reduced to a single spin in a field.

The subscript in $\Asym$ and $\Ksym$ records that the rates are equal. It is
kept even though no other case appears here, because the symmetric urn is one
point of a one-parameter family in which the spin axis tilts with the bias, as
described in \cref{sec:teaser}.

\subsection{Spectral decoding: Krawtchouk vectors and the rotated spin basis}
\label{sub:ehr-spectral}

Before turning to the imaginary slice we make the real-slice picture more explicit. The gauged semigroup \eqref{eq:gauged-semigroup} is diagonalised by the
eigenvectors of $\Asym$, so the whole relaxation of the urn is a finite sum of
exponentials weighted by eigenvector overlaps. For the Ehrenfest chain these
eigenvectors are classical objects, the Krawtchouk polynomials
\cite{Krawtchouk,KarlinMcGregor}, and the spin identity \eqref{eq:Asym-Jx} explains
them in one line.
\begin{proposition}[Krawtchouk eigenvectors as a rotated spin basis]
\label{prop:ehr-kraw}
Let $u_m$, $m=-\tfrac n2,\dots,\tfrac n2$, be the orthonormal eigenvectors of
$\Asym$ with eigenvalue $2m/n$. Then the following hold.
\begin{enumerate}[leftmargin=1.7em,itemsep=2pt,topsep=2pt]
\item $u_m$ is the $J_x$ eigenvector of eigenvalue $m$. In the $J_z$ basis its
entries are, up to a column phase, the entries of the spin-$\tfrac n2$ rotation by
$\tfrac\pi2$ about the $y$ axis,
\[
u_m(k)=\bra{k-\tfrac n2}\,e^{-i\tfrac\pi2 J_y}\,\ket m,
\]
the Wigner rotation matrix at $\tfrac\pi2$, which is real.
\item The rescaled entries $\mathrm K_m(k):=u_m(k)/\sqrt{\nu_k}$ are polynomials of
degree $\tfrac n2-m$ in $k$, the (normalised, symmetric) Krawtchouk polynomials,
orthogonal with respect to the binomial weight $\nu$,
$\sum_k\nu_k\,\mathrm K_m(k)\mathrm K_{m'}(k)=\delta_{mm'}$.
\item The semigroup has the exact spectral (Karlin--McGregor) form
\begin{equation}
\bigl(e^{tQ}\bigr)_{kl}
=\frac{\sqrt{\nu_l}}{\sqrt{\nu_k}}\sum_{m=-n/2}^{n/2}
e^{(2m-n)t}\,u_m(k)\,u_m(l)
=\nu_l\sum_{m}e^{(2m-n)t}\,\mathrm K_m(k)\,\mathrm K_m(l).
\label{eq:ehr-KM}
\end{equation}
\end{enumerate}
\end{proposition}

\begin{proof}
(i) The rotation $R=e^{-i\tfrac\pi2 J_y}$ rotates the $z$ axis to the $x$ axis, so
conjugation by $R$ gives
$R\,J_z\,R^{*}=J_x$ in every spin representation \cite{Sakurai,Hall}. Hence the
columns of $R$ in the $J_z$ basis, which are $R\ket m$, are eigenvectors of $J_x$
with eigenvalue $m$, and they are real because $iJ_y$ is a real matrix. By
\eqref{eq:Asym-Jx} they are the eigenvectors of $\Asym$.
(ii) The eigenvector equation $\Asym u_m=\tfrac{2m}n u_m$, written for the
rescaled entries $\mathrm K_m(k)=u_m(k)/\sqrt{\nu_k}$, is exactly the three-term
recurrence of the birth-death chain,
\[
(n-k)\,\mathrm K_m(k+1)+k\,\mathrm K_m(k-1)=2m\,\mathrm K_m(k),
\]
which is solved by polynomials in $k$ of the stated degrees, the Krawtchouk family
\cite{Krawtchouk,KarlinMcGregor}. Orthogonality with weight $\nu_k$ is the
orthonormality of the $u_m$ transported by the rescaling. (iii) Expand
$e^{tQ}=D^{-1}e^{s(\Asym-\Id)}D$ with $s=nt$ (the transpose of
\eqref{eq:gauged-semigroup}) in the eigenbasis, and use $s(\tfrac{2m}n-1)=(2m-n)t$.
\end{proof}

The classical fact that the Ehrenfest eigenvectors are Krawtchouk polynomials is,
from this viewpoint, the statement that the eigenvectors of $J_x$ in the $J_z$
basis are the entries of a spin-$\tfrac n2$ rotation by $\tfrac\pi2$. The spectral
sum \eqref{eq:ehr-KM} is the real-slice decoding made concrete. Each term is one
boost plane of \cref{rem:boost-planes}, and the overlap coefficients are rotation
matrix entries.

\begin{example}[The $n=2$ urn decoded spectrally, by hand]
\label{ex:ehr-spectral-n2}
Start the $n=2$ urn with both balls in $a$, that is $\pi_0=(0,0,1)^{T}$, and use
the eigenvectors of \cref{ex:ehr-small}. The gauged initial vector expands as
\[
D^{-1}\pi_0=(0,\,0,\,2)^{T}=u_{+1}-\sqrt2\,u_{0}+u_{-1},
\]
and the three modes decay with $e^{s(\alpha-1)}$ at $s=2t$, that is with factors
$1$, $e^{-2t}$, $e^{-4t}$. Applying $D$ to the sum
$u_{+1}-\sqrt2\,e^{-2t}u_{0}+e^{-4t}u_{-1}$ gives, entry by entry,
\[
\pi_t=\Bigl(
\tfrac14\bigl(1-e^{-2t}\bigr)^{2},\ \
\tfrac12\bigl(1-e^{-4t}\bigr),\ \
\tfrac14\bigl(1+e^{-2t}\bigr)^{2}
\Bigr)^{T},
\]
which is exactly the binomial law \eqref{eq:ehr-binom} with
$p(t)=\tfrac12(1+e^{-2t})$, since $\pi_t(2)=p^{2}$, $\pi_t(1)=2p(1-p)$, and
$\pi_t(0)=(1-p)^{2}$. The three spectral terms are the three boost planes of the
$n=2$ flow, with rapidities $s\alpha_m=2mt$ for $m=-1,0,1$, and the mode $m=+1$ is
the stationary plane whose boost exactly cancels the damping $e^{-s}$.
\end{example}

\subsection{The Born rule and the quantum clock}
\label{sub:ehr-born}

The imaginary slice attached to \eqref{eq:Ksym} is, in each ancilla sector, the
register spin rotation $e^{-i\theta J_x}$. Indeed, on the $\sigmay$ eigenvector
$\ket{y_\pm}$ the slice $U_\theta=e^{i\theta\Ksym}$ acts on the register as
$e^{\pm i\theta\frac2n J_x}$, by \cref{prop:imag-spectrum}, so up to the fixed
rescaling $\theta\mapsto\tfrac2n\theta$ of the slice parameter the register motion
is the rotation $e^{-i\theta J_x}$. We measure it by the Born rule and recover the
classical law. Write $\ket{m}$ for the $J_z$ eigenstate, $m=k-\tfrac n2$.

\begin{lemma}[Spin-$\tfrac n2$ Born rule for an axial start]
\label{lem:ehr-born}
For all $\theta\in\RR$ and $k\in\{0,\dots,n\}$,
\begin{equation}
\bigl|\bra{m=k-\tfrac n2}\,e^{-i\theta J_x}\,\ket{m=+\tfrac n2}\bigr|^{2}
=\binom nk\cos^{2k}\!\tfrac\theta2\,\sin^{2(n-k)}\!\tfrac\theta2 .
\label{eq:ehr-bornlaw}
\end{equation}
\end{lemma}

\begin{proof}
The spin-$\tfrac n2$ representation is the symmetric subspace of $n$ qubits, in
which $J_x$ is the restriction of $\tfrac12\sum_{j}\sigmax_{(j)}$ and the highest
weight is $\ket{m=+\tfrac n2}=\ket{+\zhat}^{\otimes n}$ \cite{Schwinger,Sakurai}.
Hence
\[
e^{-i\theta J_x}\ket{m=+\tfrac n2}
=\bigl(e^{-i\theta\sigmax/2}\ket{+\zhat}\bigr)^{\otimes n}
=\ket{\nhat(\theta)}^{\otimes n},
\]
the $n$-fold tensor power of the one-qubit state at Bloch polar angle $\theta$. A
single qubit gives $|\braket{+\zhat}{\nhat(\theta)}|^{2}=\cos^{2}\tfrac\theta2$ and
$|\braket{-\zhat}{\nhat(\theta)}|^{2}=\sin^{2}\tfrac\theta2$, so the number of
qubits read as $+\zhat$ is $\mathrm{Binomial}(n,\cos^{2}\tfrac\theta2)$, and
reading $k$ of them is \eqref{eq:ehr-bornlaw}.
\end{proof}

\begin{theorem}[Classical law equals Born law, under a nonlinear clock]
\label{thm:born}
For the symmetric chain, every $n\ge1$, $t\ge0$, and $k$,
\begin{equation}
\pi_t(k)=\bigl|\bra{m=k-\tfrac n2}\,e^{-i\theta(t)J_x}\,\ket{m=+\tfrac n2}
\bigr|^{2},
\qquad \theta(t)=\arccos\!\bigl(e^{-2t}\bigr).
\label{eq:born-pi}
\end{equation}
\end{theorem}

\begin{proof}
With $\cos\theta=e^{-2t}$ the half-angle identity gives
$\cos^{2}\tfrac\theta2=\tfrac{1+e^{-2t}}2=p(t)$ and
$\sin^{2}\tfrac\theta2=1-p(t)$. Substituting into \eqref{eq:ehr-bornlaw}
reproduces \eqref{eq:ehr-binom}.
\end{proof}

The Ehrenfest distribution is the Born distribution of a spin rotation. The only
nontrivial ingredient is the clock $\theta(t)$, which has two independent
readings.

\begin{remark}
The reason $\theta$, the quantum clock, cannot be
proportional to $t$ is not analytic subtlety. It is that the two read-outs have
different degrees. On the real slice the law is recovered linearly in $W_s$, by
a marginal with a damping factor. On the imaginary slice it is recovered
quadratically, by the Born rule. A linear-in-$W$ quantity and a quadratic-in-$W$
quantity cannot agree under the identity reparametrisation, and
$\theta(t)=\arccos(e^{-2t})$ is exactly the change of variables that repairs the
mismatch for this very simple chain.
\end{remark}

\begin{proposition}[Two readings of the clock]
\label{prop:clock}
For the symmetric chain, $\cos\theta(t)=e^{-2t}$ is the relaxation factor of the
slowest non-constant mode, whose generator eigenvalue is $-2$ by
\cref{cor:ehr-spec}. Moreover $\theta(t)/2$ is a statistical angle. If
$q_t=(p(t),1-p(t))$ is the law of one ball and $\delta_a=(1,0)$, then with the
Bhattacharyya coefficient $\BCo(p,q)=\sum_i\sqrt{p_iq_i}$ and the statistical
angle $\arccos\BCo$ (\cref{sec:primer-info}),
\[
\BCo(\delta_a,q_t)=\sqrt{p(t)},\qquad
\arccos\BCo(\delta_a,q_t)=\arccos\sqrt{p(t)}=\frac{\theta(t)}2 .
\]
\end{proposition}

\begin{proof}
The spectral statement is \cref{cor:ehr-spec}. For the geometric one,
$\BCo(\delta_a,q_t)=\sqrt{p(t)}$ directly, and
$\cos^{2}\tfrac{\theta(t)}2=p(t)$ from \cref{thm:born} gives
$\cos\tfrac{\theta(t)}2=\sqrt{p(t)}$.
\end{proof}

The clock is therefore not an arbitrary time reparametrisation. It is the geodesic
distance the single-ball law has travelled in the Fisher--Rao geometry, which we
discuss in \cref{sec:geometry-info} and in the primers
(\cref{sec:primer-geometry,sec:primer-info}). We use the convention that the
Fisher--Rao distance between two laws is $2\arccos\BCo$, the length of the
great-circle arc between their amplitude vectors $\sqrt p$ and $\sqrt q$, so that
the clock $\theta(t)$ is that distance and the half-angle
$\theta(t)/2=\arccos\BCo$ is the Bhattacharyya angle.

\begin{proposition}[Symmetric binomial initial laws]
\label{prop:born-binomial-starts}
Suppose the symmetric urn starts from $\mathrm{Binomial}(n,p_0)$, where
$p_0=\cos^{2}(\theta_0/2)$ and $\theta_0\in[0,\pi]$. Then its law at time $t$ is the
Born distribution of the coherent state
\begin{equation}
e^{-i\theta(t)J_x}\ket{m=+\tfrac n2},\qquad
\theta(t)=\arccos\bigl(\cos\theta_0\,e^{-2t}\bigr).
\label{eq:born-binomial-clock}
\end{equation}
Equivalently, the initial coherent state at angle $\theta_0$ is rotated through
the additional angle $\theta(t)-\theta_0$ about the same axis.
\end{proposition}

\begin{proof}
Independence of the balls gives a binomial law with parameter
$p(t)=\tfrac12+(p_0-\tfrac12)e^{-2t}$. By the half-angle identity and the
definition of $\theta_0$,
\[
\cos^{2}\!\frac{\theta(t)}2
=\frac{1+\cos\theta_0\,e^{-2t}}2
=\frac12+\left(p_0-\frac12\right)e^{-2t}=p(t).
\]
The Born formula \eqref{eq:ehr-bornlaw} therefore gives the same binomial law.
The final statement follows from the group property of rotations about $J_x$.
\end{proof}

\subsection{The Fisher--Rao lift}
\label{sub:q1-objects}

For a probability law $\pi$ on $\{0,\dots,n\}$ write $\ket{\sqrt\pi}$ for the
column vector with components $\sqrt{\pi_k}$. For the Ehrenfest urn set

\begin{equation}
\sFR_t:=\ketbra{\sqrt{\pi_t}}{\sqrt{\pi_t}},\qquad
\bigl(\sFR_t\bigr)_{kl}=\sqrt{\pi_t(k)\,\pi_t(l)},
\label{eq:sfr-def}
\end{equation}
the rank-one projector onto the amplitude vector of the law at time $t$. This
matrix is real symmetric, positive semidefinite, and has unit trace because
$\sum_k\pi_t(k)=1$. It is therefore a legitimate density matrix for every
$t$, on the register space $\CC^{M}$ alone, and it is defined with no
dynamical input at all. It is the square-root embedding of the simplex, the
one that carries the Fisher--Rao geometry,
applied pointwise to the solution curve $t\mapsto\pi_t$.

The second object is our familiar imaginary-slice pseudo-density. For
the physical embedding the initial pseudo-wave $\Phi_0$ is real, the slice
matrices $W_{i\theta}=U_\theta$ are real orthogonal, so
$\Phi_{i\theta}=U_\theta\Phi_0$ stays real and
\begin{equation}
\rho_{i\theta}=\Phi_{i\theta}\Phi_{i\theta}^{T}
\label{eq:rho-imag}
\end{equation}
is real symmetric positive semidefinite with the constant trace
$\Phi_0^{T}\Phi_0=1+\chi^{2}(\pi_0\,\|\,\nu)$. Dividing once by this constant produces a
genuine density matrix on the doubled space $\CC^{2M}$ that evolves by the
real orthogonal conjugation $\rho\mapsto U_\theta\rho\,U_\theta^{T}$. For
the axial start $\pi_0=\delta_n$ of the urn the constant is
$1/\nu_n=2^{n}$ and the normalised initial lift is simply the basis vector
$(e_n,0)$.

Both objects are rank one, both are density matrices after at most one
normalisation, both are real symmetric, and both have constant purity. The
substance of the comparison is in the motion and in the readout, and the
Ehrenfest chain turns out to make the comparison exact.

\paragraph{An exact coincidence along the Ehrenfest orbit}
\label{sub:q1-coincidence}

\begin{proposition}[The Fisher--Rao lift is a rigid spin rotation on the
clock]
\label{prop:fr-rotation}
For the symmetric urn started at $\pi_0=\delta_n$, for every $t\ge0$ and
every $k$,
\begin{equation}
\sqrt{\pi_t(k)}=\bra{k}\,e^{-i\theta(t)\Jy}\,\ket{n},\qquad
\theta(t)=\arccos\bigl(e^{-2t}\bigr),
\label{eq:fr-identity}
\end{equation}
and the right side is real and nonnegative. Consequently
$\sFR_t=e^{-i\theta(t)\Jy}\,\ketbra{n}{n}\,e^{i\theta(t)\Jy}$, so the
Fisher--Rao lift of the relaxing law is itself a rigid rotation of a spin
coherent state, about the $y$ axis, run on exactly the Ehrenfest quantum
clock.
\end{proposition}

\begin{proof}
For one qubit, $\sigmay\ket{+\hat z}=i\ket{-\hat z}$ gives
\[
e^{-i\theta\sigmay/2}\ket{+\hat z}
=\cos\tfrac\theta2\,\ket{+\hat z}+\sin\tfrac\theta2\,\ket{-\hat z},
\]
with real nonnegative amplitudes for $\theta\in[0,\tfrac\pi2]$. Taking the
$n$-fold tensor power, exactly as in the proof of \cref{lem:ehr-born} but
with the $y$ rotation in place of the $x$ rotation, the amplitude on the
symmetric weight-$k$ state is
$\sqrt{\binom nk}\,\cos^{k}\tfrac\theta2\,\sin^{n-k}\tfrac\theta2$. With
$\cos^{2}\tfrac{\theta(t)}2=p(t)$, which is the clock identity in the proof
of \cref{thm:born}, this amplitude equals
$\sqrt{\binom nk\,p(t)^{k}(1-p(t))^{n-k}}=\sqrt{\pi_t(k)}$.
\end{proof}

 At
$\theta=\tfrac\pi2$ the entries are $2^{-n/2}\sqrt{\binom nk}=\sqrt{\nu_k}$,
so the orbit ends at the amplitude vector of the equilibrium law, which is
simultaneously the Perron eigenvector $\sqrt\nu$ of the gauge
and the extremal column of the Krawtchouk rotation
$e^{-i\pi\Jy/2}$.

\begin{proposition}[One fixed diagonal gauge links the two lifts]
\label{prop:gauge-link}
Let $\Ephase:=\diag\bigl(i^{\,n-k}\bigr)_{k=0}^{n}$, a fixed unitary diagonal
matrix independent of $t$, and let $\ket{\psi_\theta}:=e^{-i\theta\Jx}\ket n$ be
the register coherent state of the imaginary slice. Then
\begin{equation}
\Ephase\,\Jx\,\Ephase^{-1}=\Jy,\qquad \Ephase\ket n=\ket n,\qquad
\Ephase\ket{\psi_{\theta}}=e^{-i\theta\Jy}\ket{n},
\label{eq:gauge-link}
\end{equation}
so $\ket{\sqrt{\pi_t}}=\Ephase\ket{\psi_{\theta(t)}}$ and
$\sFR_t=\Ephase\,\ketbra{\psi_{\theta(t)}}{\psi_{\theta(t)}}\,\Ephase^{*}$.
\end{proposition}

\begin{proof}
$\Ephase$ is diagonal, so conjugation only rephases the bands. On the lower
band, with $a:=n-k$,
\[
\bigl(\Ephase\Jx\Ephase^{-1}\bigr)_{k+1,k}
=i^{\,a-1}\,(\Jx)_{k+1,k}\,i^{-a}=-\,i\,(\Jx)_{k+1,k}=(\Jy)_{k+1,k},
\]
and on the upper band the same computation gives the factor $+i$, which
matches $(\Jy)_{k,k+1}=i\,(\Jx)_{k,k+1}$. The entry $\Ephase_{nn}=i^{0}=1$
fixes $\ket n$. Exponentiating the first identity gives
$\Ephase e^{-i\theta\Jx}\Ephase^{-1}=e^{-i\theta\Jy}$, and applying both
sides to $\ket n$ gives the third identity. The final statement is
\cref{prop:fr-rotation}.
\end{proof}

The content of \cref{prop:gauge-link} is that the Fisher--Rao lift and the
imaginary-slice register state are the same object up to one fixed diagonal
phase matrix. The phases $i^{\,n-k}$ do not depend on $t$. Since $\Ephase$
is diagonal it does not touch diagonal entries, which is why both objects
hand back the law, $\sFR_t$ by inspection,
$\bra k\sFR_t\ket k=\pi_t(k)$, and $\ket{\psi_{\theta(t)}}$ through the Born
rule of \cref{thm:born}. The lift $\ket{\sqrt{\pi_t}}$ is exactly the
positive representative of the coherent state, the vector obtained by
stripping the phases, and along this orbit stripping phases costs nothing
because the phase pattern is rigid.

\begin{corollary}[Classical overlaps equal quantum overlaps along the orbit]
\label{cor:bc-fidelity}
For all $t_1,t_2\ge0$,
\begin{equation}
\BCo(\pi_{t_1},\pi_{t_2})=\braket{\sqrt{\pi_{t_1}}}{\sqrt{\pi_{t_2}}}
=\braket{\psi_{\theta(t_1)}}{\psi_{\theta(t_2)}}
=\cos^{n}\!\Bigl(\tfrac{\theta(t_2)-\theta(t_1)}2\Bigr).
\label{eq:bc-fidelity}
\end{equation}
The Bhattacharyya overlap of the laws at two times equals the fidelity
overlap of the corresponding register states, and it depends only on the
clock increment.
\end{corollary}

\begin{proof}
The first equality is the definition of $\BCo$. The second is
\cref{prop:gauge-link} and unitarity of $\Ephase$. The third factorises over
the $n$ qubits, one factor
$\bra{+\zhat}e^{\,i(\theta(t_1)-\theta(t_2))\sigmax/2}\ket{+\zhat}
=\cos\tfrac{\theta(t_2)-\theta(t_1)}2$ per ball.
\end{proof}

\paragraph{The speed of each motion}
\label{sub:q1-speed}

Differentiating $\cos\theta(t)=e^{-2t}$ gives
\begin{equation}
\dot\theta(t)=\frac{2e^{-2t}}{\sqrt{1-e^{-4t}}}=2\cot\theta(t),
\label{eq:clock-ode}
\end{equation}
so \cref{prop:fr-rotation} can be stated as a differential equation,
\begin{equation}
\frac{d\sFR_t}{dt}=-\,2i\,\cot\theta(t)\,\bigl[\Jy,\sFR_t\bigr].
\label{eq:fr-vn}
\end{equation}

This is a von Neumann equation with a fixed axis and a time-dependent
speed. The commutator $-i[\Jy,\cdot\,]$ preserves real symmetric matrices
because $i\Jy$ is real, so \eqref{eq:fr-vn} is an honest flow on real
states. The speed $2\cot\theta$ is infinite at $t=0$, which is the familiar
statement that a law leaves a simplex vertex at unbounded statistical speed,
the Bhattacharyya angle growing like $\sqrt t$, and the speed decays to zero
as $\theta\to\tfrac\pi2$, so the motion freezes at equilibrium. The
imaginary-slice state obeys the autonomous equation
$d\rho/d\theta=-i[\HNUO,\rho]$ at unit speed in $\theta$
and never freezes. The two motions have the same rigid orbit geometry and
differ only in parametrisation, and the clock is precisely the
reparametrisation that glues them.

\subsection{Three exact connections}
\label{sub:ehr-connections}

The spin form \eqref{eq:Ksym} exposes three exact correspondences between the urn
and objects from far afield. Each is elementary once the spin identity is in hand.

\paragraph{Connection I: relaxation modes are Lorentz boosts}
By \cref{thm:dilation} and \cref{rem:boost-planes} the real slice $W_s$ acts on
the eigenplane of $A$ with eigenvalue $\alpha_m=2m/n$ as the two-by-two block
$\beta_{s\alpha_m}$ of \cref{sub:primer-boost}, a Lorentz boost of rapidity
$s\alpha_m=2mt$. The real-time Ehrenfest flow is therefore a direct sum of boosts,
one per mode, and the spectrum of the gauge is a spectrum of rapidity rates.
Relaxation to equilibrium is each mode boosting at its own constant rate
$\alpha_m$, the slowest survivor being the stationary mode $\alpha=1$ whose boost
exactly cancels the damping. Composition of the chain over disjoint time intervals
is addition of rapidities, the relativistic velocity-addition law in disguise. For
$n=2$ the three planes and their rapidities $2mt$, $m=-1,0,1$, were listed
explicitly in \cref{ex:ehr-spectral-n2}.

\paragraph{Connection II: the imaginary slice is a depth-one quantum circuit}
In the symmetric subspace of $n$ qubits, $J_x=\tfrac12\sum_j\sigmax_{(j)}$, so the
register rotation factorises as
\[
e^{-i\theta J_x}=\bigotimes_{j=1}^{n}e^{-i\theta\sigmax/2}
=\bigotimes_{j=1}^{n}R_x(\theta).
\]
This is a circuit of depth one, a single layer of identical one-qubit
$x$-rotations acting on disjoint wires with no entangling gates. The recovery of
$\pi_t$ is then a quantum sampling procedure. Prepare $n$ qubits in
$\ket{0}^{\otimes n}$, apply $R_x(\theta(t))$ to each, and measure in the
computational basis. Since $\ket 0=\ket{+\zhat}$ and each qubit is one ball, the
number of qubits read as $\ket 0$ is the number of balls in urn $a$, so that
count is distributed as the Ehrenfest law $\pi_t$ by \cref{thm:born}. The
complementary count, the number of qubits read as $\ket 1$, is the occupancy of
urn $b$ and follows $\mathrm{Binomial}(n,1-p(t))$. Being a product circuit it is
classically simulable in linear time, so the interest is structural rather than
computational. There is no quantum speedup here.

\paragraph{Connection III: relaxation is a rigid glide of Majorana stars}
A spin-$\tfrac n2$ state corresponds, by the Majorana representation
\cite{Majorana}, to $n$ points (stars) on the Bloch sphere, and a spin rotation
moves every star rigidly. The start $\ket{m=+\tfrac n2}=\ket{+\zhat}^{\otimes n}$
is the coherent state with all $n$ stars stacked at the north pole. Under
$e^{-i\theta J_x}$, an $\SO(3)$ rotation about the $x$ axis, the $n$ stars stay
coincident and slide together along a meridian to polar angle $\theta$. As $t$
grows the clock $\theta(t)=\arccos(e^{-2t})$ runs from $0$, all balls in $a$ and
all stars at the pole, to $\tfrac\pi2$, equilibrium and stars on the equator.
Ehrenfest relaxation is a rigid stellar glide down a quarter meridian.

% ======================================================================

\section{Geometry and information}
\label{sec:geometry-info}

We now record a few geometric and information-theoretic structures of the PQR for
a general reversible chain. The central point is a clean separation between two
geometries on the probability simplex, the geometry of inference and the geometry
of convergence, with the PQR singling out the second. Some basic references here
are Amari and Nagaoka \cite{AmariNagaoka} and Bengtsson and \.Zyczkowski
\cite{BengtssonZyczkowski}, and the differential-geometric and
information-theoretic vocabulary used below (metrics, pullbacks, geodesic
distance, Fisher information, divergences) is collected in
\cref{sec:primer-geometry,sec:primer-info}.

\subsection{Two information geometries on the simplex}
\label{sub:two-geometries}

There are two natural ways to turn a law $p$ on $\Omega$ into an amplitude vector.
The first is the square-root map $p\mapsto\sqrt p=(\sqrt{p_i})_i$, which lands on
the unit sphere. The second is the pseudo-wave map of the PQR,
\begin{equation}
\psi_p:=D^{-1}p=\Bigl(\tfrac{p_i}{\sqrt{\nu_i}}\Bigr)_i,
\label{eq:pseudo-wave}
\end{equation}
which is linear in $p$. Each carries a Euclidean geometry back to the simplex, and
the two are genuinely different.

\paragraph{The square-root map and the geometry of inference}
The image $\sqrt p$ lies on the unit sphere, and the round metric pulls back to
the Fisher--Rao metric, whose geodesic distance is the statistical angle
$\arccos\BCo(p,q)$ with $\BCo(p,q)=\sum_i\sqrt{p_iq_i}$ (up to the factor $2$
convention fixed in \cref{sub:ehr-born}). By the theorem of Chentsov
\cite{Cencov,AmariNagaoka}, the Fisher--Rao metric is, up to scale, the unique
Riemannian metric on the interior of the simplex that contracts under every
stochastic map. It is therefore the intrinsic geometry of statistical
distinguishability, the one that governs estimation through the Cram\'er-Rao bound
\cite{Rao} and testing through the Bhattacharyya coefficient \cite{Bhattacharyya}.

\paragraph{The pseudo-wave map and the geometry of convergence}
The pseudo-wave \eqref{eq:pseudo-wave} carries the flat geometry of
$L^{2}(1/\nu)$, with inner product
$\langle u,v\rangle_{1/\nu}=\sum_iu_iv_i/\nu_i$. Two facts make it the natural
geometry for relaxation towards equilibrium.

\begin{proposition}[The chain is a self-adjoint contraction in $L^{2}(1/\nu)$, and
$\chi^{2}$ is a squared distance]
\label{prop:chi2}
Under the pseudo-wave map the semigroup acts by
$\psi_{\pi_t}=D^{-1}T_tD\,\psi_{\pi_0}=e^{\Lambda t(A-\Id)}\psi_{\pi_0}$ with
$A=A^{T}$, so $T_t$ is self-adjoint in $L^{2}(1/\nu)$. The stationary law maps to
the unit vector $\psi_\nu=\sqrt\nu$, and the $\chi^{2}$ divergence is the
corresponding squared distance to it,
\begin{equation}
\chi^{2}(p\,\|\,\nu)=\sum_i\frac{(p_i-\nu_i)^{2}}{\nu_i}=\|\psi_p-\psi_\nu\|^{2}.
\label{eq:chi2-dist}
\end{equation}
Consequently, with $\gamma$ the spectral gap of the chain,
$\chi^{2}(\pi_t\,\|\,\nu)\le e^{-2\gamma t}\,\chi^{2}(\pi_0\,\|\,\nu)$.
\end{proposition}

\begin{proof}
The action is \eqref{eq:gauged-semigroup} conjugated by $D$, and $A-\Id$ is
symmetric, so $e^{\Lambda t(A-\Id)}$ is self-adjoint for the standard inner
product, which is the $L^{2}(1/\nu)$ inner product after the $D^{-1}$ change of
variables. Equation \eqref{eq:chi2-dist} is the identity
$\sum_i(p_i/\sqrt{\nu_i}-\sqrt{\nu_i})^{2}=\sum_i(p_i-\nu_i)^{2}/\nu_i$. The decay
follows because $\psi_p-\psi_\nu$ lies in the orthogonal complement of
$\psi_\nu$, on which the symmetric operator $\Lambda(A-\Id)$ has largest
eigenvalue at most $-\gamma$, so $\|e^{\Lambda t(A-\Id)}\|\le e^{-\gamma t}$ there
and the squared distance decays like $e^{-2\gamma t}$.
\end{proof}

The two geometries agree at $\nu$, where both reduce to the same quadratic form,
and they differ globally. The square-root, Fisher--Rao geometry is the geometry of
static inference, distinguished by Chentsov uniqueness. The pseudo-wave,
$\chi^{2}$ geometry is the geometry of dynamic convergence, distinguished by
linearising and self-adjointing the chain. The PQR is built on the second map,
which is why it sees relaxation so cleanly. The two geometries met once already,
in the Ehrenfest clock of \cref{prop:clock}, whose reading as a Bhattacharyya
angle is Fisher--Rao and whose reading as $e^{-2t}$ is spectral, that is
$\chi^{2}$.

\subsection{The bilinear state space is a non-compact quadric}
\label{sub:quadric}

The pseudo-wave of \cref{sec:pseudostate} lives on a quadric, by
\cref{rem:noncompact-bloch}. We record its shape. Write $N=2M$ and let
$\Quad=\{\Phi\in\CC^{N}:\Phi^{T}\Phi=1\}$ be the affine complex quadric of unit
bilinear length.

\begin{proposition}[The quadric is the tangent bundle of a sphere]
\label{prop:quadric}
The real and imaginary parts of $\Phi=x+iy\in\Quad$ satisfy
$|x|^{2}-|y|^{2}=1$ and $x\cdot y=0$, so $\Quad$ is the set of pairs $(x,y)$ with
$x$ on a sphere of radius $|x|\ge1$ and $y$ tangent to it. After rescaling,
$\Quad$ is the tangent bundle $TS^{N-1}$. On this quadric the imaginary slice
$W_{i\theta}$ acts as a real rotation of the base sphere, and the real slice
$W_s$ acts as a boost along the fibres, growing $|x|$ and $|y|$ without bound
while preserving $|x|^{2}-|y|^{2}=1$.
\end{proposition}

\begin{proof}
For $\Phi=x+iy$, $\Phi^{T}\Phi=(|x|^{2}-|y|^{2})+2i\,x\cdot y$, so
$\Phi^{T}\Phi=1$ is exactly $|x|^{2}-|y|^{2}=1$ and $x\cdot y=0$. The slice
statements are the block form \eqref{eq:Wz}, which on a real base vector is a
rotation at $z=i\theta$ and a hyperbolic boost at $z=s$.
\end{proof}

The contrast with quantum mechanics is exactly the one presented in the
introduction. The compact Bloch sphere of pure quantum states is replaced by the
non-compact quadric $TS^{N-1}$, and relaxation is the projective approach to its
boundary at infinity, the null quadric $\Phi^{T}\Phi=0$, whose decoded image is
the stationary law. The tangent-bundle language, and the sense in which a quadric
``is'' a tangent bundle, are recalled in \cref{sec:primer-geometry}.

\subsection{Mixtures and the conserved label spectrum}
\label{sub:mixtures}

The pseudo-density extends to mixtures. Given an ensemble of initial laws
$\pi^{(\ell)}$ with weights $w_\ell\ge0$, $\sum_\ell w_\ell=1$, the pseudo-mixed
state is the convex combination
$\rho^{E}_z=\sum_\ell w_\ell\,\Phi^{(\ell)}_z(\Phi^{(\ell)}_z)^{T}$ of the
pseudo-densities of \cref{sec:pseudostate}.

\begin{proposition}[Mixtures evolve as mixtures, with frozen weights]
\label{prop:mix-flow}
$\rho^{E}_z=W_z\rho^{E}_0W_z^{-1}$, so a pseudo-mixed state is exactly the
lossless lift of a reversible chain started from the random law $\pi^{(\ell)}$
with probability $w_\ell$. The weights are constants of the motion.
\end{proposition}

\begin{proof}
Each term evolves by $\Phi^{(\ell)}_z=W_z\Phi^{(\ell)}_0$, so
$\Phi^{(\ell)}_z(\Phi^{(\ell)}_z)^{T}
=W_z\Phi^{(\ell)}_0(\Phi^{(\ell)}_0)^{T}W_z^{T}=W_z\,(\cdot)\,W_z^{-1}$ using
$W_z^{T}=W_z^{-1}$. Summing with the fixed weights gives the claim.
\end{proof}

The spectral content of a mixture is then completely described by an overlap
matrix, and it is conserved. For the statement we normalise each component. Since
the flow acts on each lift by the similarity of \cref{prop:mix-flow}, rescaling a
lift by a constant only rescales its term in the ensemble, so the normalisation
is a choice of representative and involves no loss of generality. We also record
that the statement may be read on the register alone. At $z=0$ the minus sheet
carries no mass, so the doubled pseudo-mixture
$\rho^{E}_0$ on $\CC^{2M}$ and the register pseudo-mixture on $\CC^{M}$ differ
only by a block of zeros and have the same nonzero spectrum. We work with the
register form.

\begin{theorem}[Spectral theorem and the conserved label spectrum]
\label{thm:gram}
Let $\phi^{(\ell)}:=\psi_{\pi^{(\ell)}}/\|\psi_{\pi^{(\ell)}}\|$ be the normalised
register lifts of the component laws, let
$\rho^{E}_0=\sum_\ell w_\ell\,\phi^{(\ell)}(\phi^{(\ell)})^{T}$ be the associated
normalised pseudo-mixture, and let $\mathcal{G}$ be the overlap matrix
$\mathcal{G}_{\ell m}=\kap(\pi^{(\ell)},\pi^{(m)})$, where
$\kap(p,q)=\langle\psi_p,\psi_q\rangle/(\|\psi_p\|\,\|\psi_q\|)$ is the
$\chi^{2}$ overlap, and let $\Wt=\diag(w_\ell)$. Then the nonzero spectrum of
$\rho^{E}_0$ equals the nonzero spectrum of the weighted overlap matrix
$G=\Wt^{1/2}\mathcal{G}\,\Wt^{1/2}$. Moreover the entire characteristic
polynomial of $\rho^{E}_z$, and every power trace $\Tr[(\rho^{E}_z)^{k}]$, is
conserved along the whole flow $z\in\CC$.
\end{theorem}

\begin{proof}
Writing the normalised lifts as the columns of $\Phi$, one has
$\rho^{E}_0=\Phi \Wt\Phi^{T}$ with $\Phi^{T}\Phi=\mathcal{G}$, and the fact that
$XY$ and $YX$ share nonzero spectra \cite[\S1.3]{HornJohnson}, applied to
$X=\Phi \Wt^{1/2}$ and $Y=\Wt^{1/2}\Phi^{T}$, gives
$\spec^{*}(\rho^{E}_0)=\spec(\Wt^{1/2}\mathcal{G}\Wt^{1/2})$. Conservation is
\cref{prop:mix-flow}, since similar matrices have the same characteristic
polynomial and $\Tr[(W_z\rho W_z^{-1})^{k}]=\Tr[\rho^{k}]$ by cyclicity.
\end{proof}

\begin{corollary}[Conserved pseudo-mixedness]
\label{cor:pseudomix}
For two components with weights $\alpha,1-\alpha$ and $\chi^{2}$ overlap $\kap$,
the pseudo-mixedness
$1-\Tr[(\rho^{E})^{2}]=2\alpha(1-\alpha)(1-\kap^{2})$ is constant along the flow.
\end{corollary}

\begin{proof}
With $\Wt=\diag(\alpha,1-\alpha)$ and
$\mathcal{G}=\left(\begin{smallmatrix}1&\kap\\ \kap&1\end{smallmatrix}\right)$,
$G$ has trace $1$ and determinant $\alpha(1-\alpha)(1-\kap^{2})$, so
$\Tr[(\rho^{E})^{2}]=(\Tr G)^{2}-2\det G=1-2\det G$, and conservation is
\cref{thm:gram}.
\end{proof}

On the imaginary slice the lift stays real, so $\rho^{E}_\theta$ is a genuine
density matrix and its von Neumann entropy is conserved, while the information an
observer can extract about the label $\ell$ from the register decays to zero at
the spectral-gap rate. The Fisher--Rao square-root construction gives the same
algebraic form with the $\chi^{2}$ overlap $\kap$ replaced by the Bhattacharyya
overlap $\BCo$, which is the mixture-level shadow of the same
inference-versus-convergence dichotomy. In the language of quantum state
geometry, the real lifts form the real, time-reversal-invariant slice of the
space of pure and mixed states, carrying the restriction of the Fubini-Study and
Bures metrics, and the flat $\chi^{2}$ structure above is the commutative shadow
of the quantum $\chi^{2}$ geometry of detailed-balance semigroups, see Petz
\cite{Petz} and \cite[Ch.~9 and 14]{BengtssonZyczkowski} for these metrics.

% ======================================================================

\section{What we leave for the complete paper}
\label{sec:teaser}

This section provides a concise preview of the full paper where
the four families below are treated in detail.

\paragraph{The asymmetric urn and the tilted spin}
Give the two urns different rates, so that each ball jumps from $a$ to $b$ at
rate $\lambda$ and back at rate $\mu$, normalised by $\lambda+\mu=2$. The
equilibrium is no longer the symmetric binomial but the skewed binomial with
success probability $\mu/2$, and one might expect the spin picture to break. It
does not. The gauged generator is again a spin operator, now taken along the
tilted axis
\[
\nhat=\Bigl(\sqrt{\lambda\mu},\ 0,\ \tfrac{\mu-\lambda}2\Bigr),
\]
which is a unit vector precisely because $\lambda+\mu=2$. Biasing the urn is a
rotation of the spin axis by a single element of $\SO(3)$ and nothing more.
The relaxation spectrum does not move at all, since conjugation by
a rotation cannot change eigenvalues, so the asymmetric urn has the same
relaxation rates $\{-2j\}$ as the symmetric one and only its equilibrium is
reweighted. This rigidity is invisible from the classical side, where the biased
generator looks like a genuinely different birth-death chain.

The Born read-out survives under the quantum clock time-frame
\[
\sin^{2}\frac{\theta(t)}2=\frac{1-e^{-2t}}{2\mu},
\]
which recovers the symmetric clock $\theta(t)=\arccos(e^{-2t})$ at
$\mu=\lambda=1$. This clock is available for every $t$ exactly when
$\mu\ge\tfrac12$. For $\mu<\tfrac12$ the right-hand side exceeds one once
$1-e^{-2t}>2\mu$, and no rotation about a fixed axis reaches the equilibrium
from the pole. Born representability of a classical trajectory is therefore a
genuine restriction on the pair consisting of the chain and the initial law,
and not an automatic consequence of the construction.

\paragraph{The two textbook queues}
The infinite-server queue $M/M/\infty$ and the single-server queue $M/M/1$ are
birth-death chains on $\{0,1,2,\dots\}$ in which the state counts waiting
customers. The first serves everyone in parallel, so the service rate grows with
the queue, and the second works at a fixed rate regardless of the backlog. Both
are reversible, with a Poisson equilibrium for the first and a geometric
equilibrium for the second. Under the construction they become two quite different quantum systems.
The gauged generator of $M/M/\infty$ is a displaced quantum harmonic oscillator,
with an equally spaced energy ladder and a spectral gap equal to the service
rate no matter what the load is. The gauged generator of $M/M/1$ is a free
quantum particle hopping on a half-line with a constant hopping amplitude, so it
has a continuous band of energies together with one isolated level, and its gap
closes exactly at load one, which is the classical recurrence threshold for the existence of invariant probability measure. A
single classical distinction, whether the service rate scales with the queue,
becomes the distinction between a confining potential and a free band.

The infinite-server queue also carries an exact Born read-out, the direct
analogue of \cref{thm:born}. Started empty, the register motion of the imaginary
slice sweeps out a family of oscillator coherent states, and the Poisson law of
the queue at time $t$ is recovered as the Born law of the coherent state at
angle $\theta(t)$ under a clock of the form $\sin^{2}(\mu\theta/2)$ proportional
to $1-e^{-\mu t}$, where $\mu$ is the service rate. This clock maps the whole
stochastic half-line $t\in[0,\infty)$ into the bounded window
$\mu\theta\in[0,\tfrac\pi3)$, just as the Ehrenfest clock maps it into
$[0,\tfrac\pi2)$. Both read-outs are classically simulable and neither provides quantum advantage in a quantum computing sense.

These two queues also force the technical point that this preliminary version
avoids. The rates of $M/M/\infty$ are unbounded, so there is no uniformisation
rate and the gauge is not a matrix but an unbounded self-adjoint operator. The
decoding of \cref{thm:decode} then has to be established on a domain rather than
by matrix identities, which is one of the technical points the complete paper
takes up.

\paragraph{The symmetric simple exclusion process}
Particles sit on the sites of a ring, at most one per site, and each neighbouring
pair exchanges its contents at a fixed rate. Reading an occupied site as a spin
pointing up turns a particle configuration into a basis vector of a spin chain.
The equilibrium is uniform on each particle-number sector, so the square-root
gauge is the identity and there is nothing to do. What one finds is that the
exclusion generator itself, with no transformation at all, is the ferromagnetic
Heisenberg $XXX$ Hamiltonian, invariant under the full rotation group of the
spin and diagonalisable by the Bethe ansatz (this connection is well known in the statistical physics literature). The real slice of the resulting
entire group is the dissipative exclusion semigroup, decomposed into Lorentz
boosts indexed by the Bethe eigenstates, and the imaginary slice is a unitary
Heisenberg quantum magnet. This is not an analogy between two models. It is the
same matrix read at real and at imaginary argument.

\paragraph{The stochastic Ising model}
This family makes the sharpest point of the four. Spins on a ring carry the
ferromagnetic Ising energy, and a single-spin-flip dynamics reversible for the
corresponding Gibbs measure can be built from different choices of flip rate.
The two textbook choices are the heat-bath (Glauber) rate and the Metropolis
rate. They share the same equilibrium, the same reversibility, and spectral gaps
that agree within a factor of two at every temperature, so classically they are
nearly interchangeable and one picks whichever is convenient. Under the
construction they are not interchangeable at all. In one dimension the heat-bath
rule produces a quantum chain that is free-fermionic, and therefore exactly
solvable, while the Metropolis rule produces an interacting quantum chain with
no such structure. A modelling choice that is almost invisible classically is,
on the quantum slice, the difference between a solvable and a generically
unsolvable system.

\paragraph{What we take from the four}
The pattern that repeats is the one this paper proves for the symmetric urn.
A familiar reversible chain sits on the real slice, a familiar quantum system
sits on the imaginary slice, and they are one entire representation rather than
two models placed side by side. Because an entire function is fixed by its
restriction to a line, a statement on either slice constrains the other, and the
two slices are studied by communities with almost disjoint toolkits, namely
mixing-time technology on one side and integrability and reflection positivity
on the other. Turning that observation into an explicit transfer of results, in
either direction, is the central open problem the construction raises.

% ======================================================================

\section{Conclusion}
\label{sec:concluding}

For every finite, irreducible, reversible continuous-time Markov chain and every
admissible uniformisation rate, the construction produces one entire
complex-orthogonal group $W_z=e^{zK}$. That single group is the whole content of
the paper. Its real slice recovers the Markov semigroup by the exact parity
decoding of \cref{thm:decode}, its imaginary slice is the finite-dimensional
unitary evolution of \cref{thm:imag}, and neither slice is more fundamental than
the other. The classical objects and the quantum objects are restrictions of one
holomorphic family, so they determine each other, and the irreversibility of the
law is located in the real-slice decoded read-out rather than in the doubled flow, which
is invertible and conserves a bilinear length.

Three things come out of that setup without further input. The pseudo-wave obeys a pseudo-Schr\"odinger equation, the
complex-symmetric pseudo-density obeys a bilinear von Neumann equation, and its
coordinates obey a pseudo-Bloch vector equation, all in
\cref{thm:bilinear,prop:pseudo-bloch}. On the imaginary slice the three collapse
onto their textbook quantum counterparts. Off it they describe motion on the
non-compact quadric of \cref{prop:quadric}, which is the geometric room in which
dissipation lives and which the compact Bloch sphere does not have. These
equations are exact and, at present, unexploited. What their conserved
quantities mean when the conserved bilinear length is complex, how their orbits
classify, and whether they are a useful computational device rather than a
structural observation are open.

Two further structures are recorded for a general finite chain. On the
probability simplex the square-root amplitude map yields the Fisher--Rao
geometry of inference and the linear pseudo-wave map yields the flat $\chi^{2}$
geometry of convergence, in which reversible relaxation is self-adjoint, and the
representation is built on the second. The mixture construction adds a conserved
label spectrum (\cref{thm:gram}).

The symmetric Ehrenfest urn makes all of this explicit. The identity
$\Asym=\tfrac2nJ_x$ turns the spectral decomposition into the Krawtchouk basis,
splits the real slice into Lorentz-boost planes, and makes the imaginary register
motion a depth-one spin rotation quantum circuit. For the axial start the classical law is
exactly the Born law under $\theta(t)=\arccos(e^{-2t})$ (\cref{thm:born}), and
symmetric binomial starts stay on the same coherent-state orbit with the clock
of \cref{prop:born-binomial-starts}. The same angle is the Fisher--Rao distance
travelled by the single-ball law, and one fixed diagonal phase gauge identifies
the Fisher--Rao lift with the spin-coherent-state orbit
(\cref{prop:gauge-link}). The Lorentz-boost, depth-one-circuit and Majorana-star
descriptions of \cref{sub:ehr-connections} are three readings of that one
identity.

% ======================================================================
%  ACKNOWLEDGEMENTS
% ======================================================================

\section*{Acknowledgements}
\addcontentsline{toc}{section}{Acknowledgements}

It is a pleasure to thank L.~R.~Fontes (IME-USP), J.~C.~A.~Barata
(IF-USP) and G.~M.~Sch\"utz (Instituto Superior T\'ecnico, Lisbon) for
fruitful discussions.

This research was supported by the S\~ao Paulo Research Foundation (FAPESP),
grant \#2023/13453-5.

% ======================================================================
\begin{appendices}

\section{A concise primer on the quantum and Lie language}
\label{sec:primer-quantum}

This appendix collects, in elementary terms, the small amount of
quantum-mechanical and Lie-theoretic vocabulary the construction uses. A reader
fluent in this material can skim it. Nothing is needed in greater generality than
appears here. Standard sources are Sakurai and Napolitano \cite{Sakurai}, Hall
\cite{Hall}, and Bengtsson and \.Zyczkowski \cite{BengtssonZyczkowski}.

\subsection{Vectors, two pairings, and two matrix groups}
\label{sub:primer-pairings}

We work in $\CC^{N}$. A column vector is written $\Phi$, and where convenient we
use Dirac notation, $\ket{\psi}$ for a vector and $\bra{\psi}$ for the conjugate
row $\psi^{*}$. Two pairings of vectors are kept strictly apart, because their
difference is the whole point of the word ``pseudo''.

The \emph{Hermitian} or sesquilinear pairing is
$\langle\Phi,\Psi\rangle=\Phi^{*}\Psi=\sum_k\bar\Phi_k\Psi_k$, conjugate-linear in
the first slot. It is the pairing of quantum mechanics. The matrices that preserve
it form the \emph{unitary group} $\UU(N)=\{U:U^{*}U=\Id\}$. The \emph{bilinear}
pairing is $(\Phi,\Psi)=\Phi^{T}\Psi=\sum_k\Phi_k\Psi_k$, linear in both slots and
carrying no complex conjugation. The matrices that preserve it form the
\emph{complex-orthogonal group} $\OO(N,\CC)=\{W:W^{T}W=\Id\}$.

These two groups are genuinely different. A real rotation belongs to both. A
complex matrix can belong to one and not the other. The central observation of
this paper is that a reversible chain sits most naturally with the bilinear
pairing while the attached quantum system sits with the Hermitian one, and that
one entire family of matrices visits both. A matrix that is complex-orthogonal
but not unitary will be called \emph{non-unitary orthogonal} (NUO). Its defining
feature is that it preserves the bilinear length $\Phi^{T}\Phi$ while changing the
Hermitian length $\Phi^{*}\Phi$.

\subsection{Hermitian generators and unitary evolution}
\label{sub:primer-unitary}

An operator $H$ is \emph{Hermitian} if $H^{*}=H$. Its eigenvalues are real, and in
quantum mechanics it represents an observable. The associated evolution is the
unitary group $U_\theta=e^{-i\theta H}$, which satisfies
$U_\theta^{*}U_\theta=\Id$ and therefore preserves the Hermitian length. A pure
state is a unit vector $\ket\psi$, and its \emph{density matrix} is the Hermitian
rank-one projector $\sigma=\ket\psi\bra\psi=\psi\psi^{*}$ with $\Tr\sigma=1$.
Under unitary evolution the density obeys the \emph{von Neumann equation}
\begin{equation}
\dot\sigma=-i[H,\sigma],\qquad [X,Y]:=XY-YX.
\label{eq:vn-quantum}
\end{equation}
The entire quantum side of this paper is \eqref{eq:vn-quantum} and its solution.
Probabilities enter through the \emph{Born rule}. If a system in the
state $\sigma$ is measured in an orthonormal basis $\{\ket k\}$, the outcome $k$
occurs with probability $\bra k\sigma\ket k$, which for a pure state $\psi$ is
$|\braket k\psi|^{2}$.

\subsection{Tensor products, Pauli matrices, and the twist}
\label{sub:primer-pauli}

For $S\in\CC^{2\times2}$ and $A\in\CC^{M\times M}$ the \emph{Kronecker product}
$S\otimes A$ is the $2M\times2M$ block matrix with blocks $S_{ab}A$. The only
rules we use are $(S\otimes A)(S'\otimes A')=SS'\otimes AA'$ and
$(S\otimes A)^{T}=S^{T}\otimes A^{T}$. The \emph{Pauli matrices} are
\[
\sigmax=\begin{pmatrix}0&1\\1&0\end{pmatrix},\qquad
\sigmay=\begin{pmatrix}0&-i\\ i&0\end{pmatrix},\qquad
\sigmaz=\begin{pmatrix}1&0\\0&-1\end{pmatrix}.
\]
Each squares to the identity. The matrices $\sigmax$ and $\sigmaz$ are real
symmetric, while $\sigmay$ is purely imaginary and antisymmetric. The diagonal
unitary $v=\diag(1,i)$ rotates $\sigmax$ into $\sigmay$, since
$v\,\sigmax\,v^{-1}=\sigmay$. This single line is the twist that turns a
hyperbolic flow into the NUO flow in \cref{sec:dilation}.

\subsection{Spin and the angular-momentum matrices}
\label{sub:primer-spin}

For each half-integer $j\in\{\tfrac12,1,\tfrac32,\dots\}$ there is a
distinguished triple of $(2j+1)\times(2j+1)$ Hermitian matrices $\Jx,\Jy,\Jz$,
the \emph{angular-momentum} or spin-$j$ operators, obeying $[\Jx,\Jy]=i\Jz$ and
its cyclic permutations. Concretely $\Jz$ is diagonal with entries
$j,j-1,\dots,-j$, and $\Jx$ is the real symmetric tridiagonal matrix whose
nonzero off-diagonal entries are $\tfrac12\sqrt{(j-m)(j+m+1)}$. We need only
three facts. First, the spin-$j$ representation is the standard quantum system of
dimension $2j+1$, see \cite[Ch.~3]{Sakurai} and \cite[Ch.~4]{Hall}. Second, for
the urn of \cref{sec:ehrenfest} the gauged generator is exactly proportional to
the spin-$n/2$ matrix $\Jx$, which is what makes that example solvable in closed
form. The exponential $e^{-i\theta\Jx}$ rotates the spin by angle $\theta$ about
the $x$ axis. Third, the rotation $e^{-i\pi\Jy/2}$ carries the $J_z$ eigenbasis to
the $J_x$ eigenbasis, which is the rotated-basis statement used in
\cref{prop:ehr-kraw}.

\subsection{Lorentz boosts of signature \texorpdfstring{$(1,1)$}{(1,1)}}
\label{sub:primer-boost}

The last object we need is the simplest hyperbolic rotation. On $\RR^{2}$
equipped with the indefinite form $x^{2}-y^{2}$, the matrices
\begin{equation}
\beta_s=\begin{pmatrix}\cosh s&\sinh s\\ \sinh s&\cosh s\end{pmatrix},\qquad
s\in\RR,
\label{eq:boost}
\end{equation}
form a one-parameter group preserving that form, in the sense that
$\beta_s^{T}\diag(1,-1)\beta_s=\diag(1,-1)$. These are the two-dimensional
\emph{Lorentz boosts}, and $s$ is the \emph{rapidity}. A boost is the hyperbolic
analogue of a rotation. A rotation by angle $\theta$ keeps points on a circle
$x^{2}+y^{2}=\text{const}$, whereas a boost of rapidity $s$ keeps them on a
hyperbola $x^{2}-y^{2}=\text{const}$ and slides them off toward infinity as
$s\to\infty$. Rapidities add under composition, $\beta_s\beta_{s'}=\beta_{s+s'}$,
which is the additivity used in Connection I of \cref{sub:ehr-connections}. We
found in \cref{sec:dilation} that the real slice of the PQR flow acts on each
relaxation mode as exactly such a boost, so that relaxation to equilibrium
becomes the rapidity running off to infinity.

\section{A concise primer on the geometric language}
\label{sec:primer-geometry}

This appendix collects the differential-geometric vocabulary used in
\cref{sec:geometry-info} and in the clock discussion of \cref{sub:ehr-born}. As
before, nothing is needed in greater generality than appears here, and every
manifold we meet is an explicit surface inside $\RR^{N}$ or $\CC^{N}$. Standard
sources are Lee \cite{Lee} and, for the physics-flavoured material, Nakahara
\cite{NakaharaGTP}.

\subsection{Metrics, lengths, and geodesic distance}
\label{sub:primer-metric}

A \emph{Riemannian metric} on a smooth surface (manifold) $\mathcal M$ assigns to
each point $x$ an inner product $g_x(\cdot,\cdot)$ on the tangent vectors at $x$,
varying smoothly with $x$. It is the structure that lets one measure the length
of a curve $c(\tau)$, as $\int\sqrt{g_{c(\tau)}(\dot c,\dot c)}\,d\tau$, and then
define the \emph{geodesic distance} between two points as the infimum of lengths
of curves joining them. A curve achieving the infimum is a \emph{geodesic}, the
analogue of a straight line. Two examples carry the whole paper.

First, Euclidean space $\RR^{N}$ with the constant inner product
$g(u,v)=\sum_iu_iv_i$. Its geodesics are straight lines and its distance is the
usual one. The pseudo-wave map of \cref{sub:two-geometries} endows the simplex
with exactly this flat geometry, transported from the weighted space
$L^{2}(1/\nu)$.

Second, the unit sphere
$S^{N-1}=\{x\in\RR^{N}:|x|=1\}$ with the \emph{round metric}, the restriction of
the Euclidean inner product to vectors tangent to the sphere. Its geodesics are
great circles, and the geodesic distance between unit vectors $x$ and $y$ is the
angle $\arccos(x\cdot y)$ between them. All statements in this paper about
statistical angles are statements about great-circle distances on a sphere.

\subsection{Pullbacks: transporting a geometry along a map}
\label{sub:primer-pullback}

If $F:\mathcal M\to\mathcal N$ is a smooth map into a space carrying a metric
$g$, the \emph{pullback metric} $F^{*}g$ on $\mathcal M$ is defined by measuring
vectors after pushing them through the derivative of $F$,
\[
(F^{*}g)_x(u,v):=g_{F(x)}\bigl(dF_x(u),\,dF_x(v)\bigr).
\]
In words, one declares the map $F$ to be distance-preserving infinitesimally and
imports the geometry of the target. This is the device used twice in
\cref{sub:two-geometries}. The square-root map $p\mapsto\sqrt p$ sends the
probability simplex into the unit sphere, and the round metric pulls back to the
Fisher--Rao metric. The pseudo-wave map $p\mapsto D^{-1}p$ sends the simplex into
flat $\RR^{M}$, and the flat metric pulls back to the $\chi^{2}$ geometry. The
two maps send the same simplex to different surfaces, which is why the two
geometries differ globally even though both are Riemannian metrics on the same
set.

A map between Riemannian spaces that preserves the metric is an \emph{isometry}.
The isometries of the round sphere are the orthogonal rotations $\OO(N,\RR)$,
which is why the imaginary slice of the PQR, a real orthogonal flow by
\cref{thm:imag}, acts on amplitude vectors as a rigid rotation. On $\RR^{2}$ with
the \emph{indefinite} quadratic form $x^{2}-y^{2}$, in which squared lengths can
be negative, the linear maps preserving the form are the Lorentz boosts of
\cref{sub:primer-boost}. The pair signature $(1,1)$ records one plus and one
minus direction. The real slice of the PQR acts on each mode plane as such a
boost.

\subsection{Tangent vectors, the tangent bundle, and the quadric}
\label{sub:primer-tangent}

A \emph{tangent vector} to a surface at a point is the velocity of a curve
passing through that point. The set of all pairs (point, tangent vector at that
point) is the \emph{tangent bundle} $T\mathcal M$. For the sphere it is
\[
TS^{N-1}=\bigl\{(x,y)\in\RR^{N}\times\RR^{N}:\ |x|=1,\ x\cdot y=0\bigr\},
\]
a smooth space of dimension $2(N-1)$, the positions on the sphere together with
all velocities along it.

This is the object that appears in \cref{prop:quadric}. The complex quadric
$\Quad=\{\Phi\in\CC^{N}:\Phi^{T}\Phi=1\}$, written in real and imaginary parts
$\Phi=x+iy$, is the set of pairs with $|x|^{2}-|y|^{2}=1$ and $x\cdot y=0$.
Rescaling $x$ to the unit sphere identifies these pairs with points of
$TS^{N-1}$, base point $x/|x|$ and tangent vector $y$ (of any length). The
identification explains the two motions seen in the paper. A rotation of
$\RR^{N}$ moves base points and tangent vectors rigidly, which is the imaginary
slice, while the real slice slides points along the fibres, growing $|x|$ and
$|y|$ together while preserving $|x|^{2}-|y|^{2}$, which is a boost in each mode
plane. The quadric is \emph{non-compact}, one can run off to infinity along the
fibres, and that non-compactness is the geometric room in which dissipation
lives. The compact Bloch sphere of quantum mechanics has no such room, which is
one more way to see that the pseudo-density of \cref{sec:pseudostate} is not a
quantum state.

\section{A concise primer on the information-theoretic language}
\label{sec:primer-info}

This appendix collects the information-theoretic vocabulary used in
\cref{sec:geometry-info,sub:ehr-born}. Standard sources are Cover and
Thomas \cite{CoverThomas} for the classical notions, Amari and Nagaoka
\cite{AmariNagaoka} for information geometry, and Bengtsson and \.Zyczkowski
\cite{BengtssonZyczkowski} and Nielsen and Chuang \cite{NielsenChuang} for the
quantum side.

\subsection{Divergences on the simplex}
\label{sub:primer-div}

A probability law on $\Omega=\{1,\dots,M\}$ is a vector $p$ with $p_i\ge0$ and
$\sum_ip_i=1$, a point of the \emph{probability simplex}. A \emph{divergence}
$\mathrm{D}(p\,\|\,q)$ is a nonnegative measure of discrepancy between laws,
vanishing iff $p=q$, not necessarily symmetric and not necessarily a distance.
Three divergences appear in this paper.

The \emph{$\chi^{2}$ divergence}
\begin{equation}
\chi^{2}(p\,\|\,q)=\sum_i\frac{(p_i-q_i)^{2}}{q_i}
\label{eq:chi2-def}
\end{equation}
is quadratic in $p$, which is why it interacts so well with the linear
pseudo-wave map. \Cref{prop:chi2} shows it is an honest squared Euclidean
distance in the geometry of $L^{2}(1/\nu)$, and
\cref{rem:false-friend-density} shows that the conserved bilinear trace of the
pseudo-density is exactly $1+\chi^{2}(\pi_0\,\|\,\nu)$.

The \emph{Kullback-Leibler divergence} (relative entropy)
$\mathrm{KL}(p\,\|\,q)=\sum_ip_i\log(p_i/q_i)$ appears only as background. It
decreases along every Markov semigroup, which is the standard statement of
dissipativity quoted in the introduction \cite{CoverThomas,LPW}. Near $p=q$ the
two agree to leading order, $\mathrm{KL}(p\,\|\,q)=\tfrac12\chi^{2}(p\,\|\,q)
+O(\|p-q\|^{3})$, so all divergences in this paper share one local quadratic
form, the Fisher information form below.

The \emph{Bhattacharyya coefficient} and the induced angle,
\begin{equation}
\BCo(p,q)=\sum_i\sqrt{p_iq_i}\in[0,1],\qquad
\angleFR(p,q)=\arccos\BCo(p,q),
\label{eq:bc-def}
\end{equation}
measure overlap rather than discrepancy \cite{Bhattacharyya}. The coefficient is
the Euclidean inner product of the amplitude vectors $\sqrt p$ and $\sqrt q$ on
the unit sphere, so the angle $\arccos\BCo$ is a great-circle (geodesic) distance
in the sense of \cref{sec:primer-geometry}. Up to normalisation it is the
\emph{Hellinger} geometry, since the Hellinger distance obeys
$\Hel^{2}(p,q)=2-2\,\BCo(p,q)$. In this paper the convention is fixed in
\cref{sub:ehr-born}. The Fisher--Rao distance is $2\arccos\BCo$, and the Ehrenfest
clock $\theta(t)$ is exactly this distance between the single-ball law and its
starting point, with the Bhattacharyya angle as its half.

\subsection{Fisher information and the Cram\'er-Rao bound}
\label{sub:primer-fisher}

Let $p_\vartheta$ be a family of laws depending smoothly on a real parameter
$\vartheta$. The \emph{Fisher information} of the family at $\vartheta$ is
\begin{equation}
\Fishinfo(\vartheta)=\sum_i\frac{\bigl(\partial_\vartheta
p_{\vartheta,i}\bigr)^{2}}{p_{\vartheta,i}} .
\label{eq:fisher-def}
\end{equation}
Its operational meaning is the \emph{Cram\'er-Rao bound} \cite{Rao,CoverThomas}.
Every unbiased estimator $\hat\vartheta$ of $\vartheta$ built from one sample of
$p_\vartheta$ has variance at least $1/\Fishinfo(\vartheta)$, and $m$ independent
samples improve this to $1/(m\Fishinfo(\vartheta))$. Large information means an
easy estimation problem. For the binomial family
$\mathrm{Binomial}(n,p(\vartheta))$ one computes
$\Fishinfo(\vartheta)=\tfrac{n}{p(1-p)}(dp/d\vartheta)^{2}$.

The Fisher information is the quadratic form of a Riemannian metric on the
simplex, the \emph{Fisher--Rao metric}, whose squared line element is
$ds^{2}=\sum_i dp_i^{2}/p_i$ (up to an overall convention factor). The
square-root map $p\mapsto\sqrt p$ turns it into the round sphere metric, since
$dp_i/\sqrt{p_i}=2\,d\sqrt{p_i}$, which is the pullback statement of
\cref{sub:two-geometries} and the source of the factor-two bookkeeping between
the angle $\arccos\BCo$ and the Fisher--Rao distance. \v{C}encov's theorem states
that this metric is, up to scale, the unique Riemannian metric on the simplex
that never expands under a stochastic map (a transition kernel applied to the
law) \cite{Cencov,AmariNagaoka}. Stochastic maps can only blur, never sharpen,
statistical distinguishability, and Fisher--Rao is the geometry that records
exactly this.

\end{appendices}

% ======================================================================


\begin{thebibliography}{99}


\bibitem{Kac} M. Kac, \emph{Random walk and the theory of Brownian motion.}
American Mathematical Monthly \textbf{54} (1947), 369--391.


\bibitem{KarlinMcGregor} S. Karlin and J. McGregor, \emph{Ehrenfest urn models.}
Journal of Applied Probability \textbf{2} (1965), 352--376.


\bibitem{Krawtchouk} M. Krawtchouk, \emph{Sur une g\'en\'eralisation des polyn\^omes
d'Hermite.} Comptes Rendus Acad. Sci. Paris \textbf{189} (1929), 620--622.


\bibitem{Schwinger} J. Schwinger, \emph{On Angular Momentum.} U.S. AEC Report NYO-3071
(1952). Reprinted in L. C. Biedenharn and H. van Dam (eds.), \emph{Quantum Theory of
Angular Momentum}, Academic Press (1965).


\bibitem{Sakurai} J. J. Sakurai and J. Napolitano, \emph{Modern Quantum Mechanics.}
2nd ed., Cambridge University Press (2017).


\bibitem{Hall} B. C. Hall, \emph{Lie Groups, Lie Algebras, and Representations.}
2nd ed., Springer (2015).


\bibitem{Aldous} D. Aldous and J. Fill, \emph{Reversible Markov Chains and Random Walks
on Graphs.} Unfinished monograph, recompiled 2014 (available online).


\bibitem{LPW} D. A. Levin, Y. Peres, and E. L. Wilmer, \emph{Markov Chains and Mixing
Times.} 2nd ed., American Mathematical Society (2017).


\bibitem{Norris} J. R. Norris, \emph{Markov Chains.} Cambridge University Press (1997).


\bibitem{Liggett_ctmc} T. M. Liggett, \emph{Continuous Time Markov Processes: An
Introduction.} American Mathematical Society (2010).


\bibitem{Majorana} E. Majorana, \emph{Atomi orientati in campo magnetico variabile.}
Il Nuovo Cimento \textbf{9} (1932), 43--50.


\bibitem{Radcliffe} J. M. Radcliffe, \emph{Some properties of coherent spin states.}
Journal of Physics A \textbf{4} (1971), 313--323.


\bibitem{BengtssonZyczkowski} I. Bengtsson and K. \.Zyczkowski, \emph{Geometry of Quantum
States.} 2nd ed., Cambridge University Press (2017).


\bibitem{Rao} C. R. Rao, \emph{Information and the accuracy attainable in the estimation
of statistical parameters.} Bulletin of the Calcutta Mathematical Society \textbf{37}
(1945), 81--91.


\bibitem{Bhattacharyya} A. Bhattacharyya, \emph{On a measure of divergence between two
statistical populations.} Bulletin of the Calcutta Mathematical Society \textbf{35}
(1943), 99--109.


\bibitem{AmariNagaoka} S. Amari and H. Nagaoka, \emph{Methods of Information Geometry.}
American Mathematical Society / Oxford University Press (2000).


\bibitem{NielsenChuang} M. A. Nielsen and I. L. Chuang, \emph{Quantum Computation and
Quantum Information.} 10th Anniversary ed., Cambridge University Press (2010).


\bibitem{NakaharaGTP} M. Nakahara, \emph{Geometry, Topology and Physics.}
2nd ed., Institute of Physics Publishing (2003).


\bibitem{Cencov} N. N. Chentsov, \emph{Statistical Decision Rules and Optimal
Inference.} Translations of Mathematical Monographs \textbf{53}, American
Mathematical Society (1982).


\bibitem{Petz} D. Petz, \emph{Monotone metrics on matrix spaces.}
Linear Algebra and its Applications \textbf{244} (1996), 81--96.


\bibitem{HornJohnson} R. A. Horn and C. R. Johnson, \emph{Matrix Analysis.}
2nd ed., Cambridge University Press (2013).


\bibitem{Lee} J. M. Lee, \emph{Introduction to Riemannian Manifolds.}
2nd ed., Springer (2018).


\bibitem{CoverThomas} T. M. Cover and J. A. Thomas, \emph{Elements of Information
Theory.} 2nd ed., Wiley (2006).


\bibitem{SzNagy} B. Sz.-Nagy, C. Foia\c{s}, H. Bercovici, and L. K\'erchy,
\emph{Harmonic Analysis of Operators on Hilbert Space.} 2nd ed., Springer (2010).


\end{thebibliography}
\end{document}